\documentclass[11pt,a4paper]{amsart}
\usepackage{amsmath}
\usepackage{amsthm}              
\usepackage{pgfplots,pdflscape,comment}
\usepackage{amssymb,bm}
\usepackage{wasysym,url}
\usepackage{enumitem}
\usepackage{color,tikz, mathdots}
\usepackage{multirow}
\usepackage{longtable}
\usepackage{color, colortbl}
\definecolor{Gray}{gray}{0.9}

\newtheorem{thm}{Theorem}[section]

\newtheorem{lemma}[thm]{Lemma}

\newtheorem{example}[thm]{Example}

\newtheorem{remark}[thm]{Remark}
\newtheorem{note}[thm]{Note}

\usetikzlibrary{arrows,decorations.pathmorphing,backgrounds,fit}

\usepackage{caption} 
\usepackage{amsaddr}
\usepackage{array}
\makeatletter
\newcommand{\thickhline}{%
    \noalign {\ifnum 0=`}\fi \hrule height 1pt
    \futurelet \reserved@a \@xhline
}
\newcolumntype{"}{@{\hskip\tabcolsep\vrule width 1pt\hskip\tabcolsep}}
\makeatother
\begin{document}

\title[The smallest bimolecular mass action CRNs admitting Andronov--Hopf bifurcation]{The smallest bimolecular mass action reaction networks admitting Andronov--Hopf bifurcation}

\author {Murad Banaji}
\address{Department of Design Engineering and Mathematics, Middlesex University London}

\author {Bal\'azs Boros}
\address{Department of Mathematics, University of Vienna}
\thanks{BB's work was supported by the Austrian Science Fund (FWF), project P32532.}

\begin{abstract}
We address the question of which small, bimolecular, mass action chemical reaction networks (CRNs) are capable of Andronov--Hopf bifurcation (from here on abbreviated to ``Hopf bifurcation''). It is easily shown that any such network must have at least three species and at least four irreversible reactions, and one example of such a network with exactly three species and four reactions was previously known due to Wilhelm. In this paper, we develop both theory and computational tools to fully classify three-species, four-reaction, bimolecular CRNs, according to whether they admit or forbid Hopf bifurcation. We show that there are, up to a natural equivalence, 86 minimal networks which admit nondegenerate Hopf bifurcation. Amongst these, we are able to decide which admit supercritical and subcritical bifurcations. Indeed, there are 25 networks which admit both supercritical and subcritical bifurcations, and we can confirm that all 25 admit a nondegenerate Bautin bifurcation. A total of 31 networks can admit more than one nondegenerate periodic orbit. Moreover, 29 of these networks admit the coexistence of a stable equilibrium with a stable periodic orbit. Thus, fairly complex behaviours are not very rare in these small, bimolecular networks. Finally, we can use previously developed theory on the inheritance of dynamical behaviours in CRNs to predict the occurrence of Hopf bifurcation in larger networks which include the networks we find here as subnetworks in a natural sense. 
\end{abstract}


\maketitle

\section{Introduction}

\subsection{Oscillation in chemical reaction networks}

Chemical reaction networks (CRNs) are interesting from both a practical and a theoretical point of view. They are central to many models in biology, and also play an important role in other areas of science and engineering. A number of powerful results, both classical and more recent, tell us about features of the dynamics of a CRN based on its combinatorial structure.

Oscillation in CRNs has been of great interest at least since the pioneering experimental work of Belousov and Zhabotinsky \cite{zhabotinsky1964periodic,belousov} in the 1950s and 1960s. However, the mathematical techniques available for determining whether a given CRN admits a (nonconstant) periodic orbit tend to be more limited than those for determining the number and nature of its equilibria. Nevertheless, there is a considerable theoretical literature on oscillation in CRNs. Papers include both those using analysis and numerics to study particular classes of networks of practical importance (\cite{dicera,Ruth1997,Kholodenko.2000aa,schustermarhlhofer,Qiao.2007aa,hellrendall2016,CONRADI2018507,obatake} are just a few of many examples); and those providing more general conditions which rule out or guarantee oscillation (examples include the original results of deficiency theory \cite{feinberg} and more recent work such as \cite{gedeonsontag2007,minchevaroussel,abphopf,errami2015,banajiCRNosci,bbhAMCrank}).

Hopf bifurcation provides a natural sufficient condition for oscillation, and consequently figures often in the work referenced above. Determining the capacity of a network for Hopf bifurcation is essentially a {\em local} problem, and potentially more tractable than other approaches to finding oscillation. Although still often challenging, the powerful machinery of bifurcation theory is at our disposal. 

Here we concern ourselves with finding the smallest bimolecular CRNs capable of Hopf bifurcation. By this we mean bimolecular networks having fewest species and reactions amongst all bimolecular networks admitting Hopf bifurcation; but we note that alternative characterisations of minimality are possible. This is discussed further in Section~\ref{secoutline} below.

By applying a range of different approaches, we are able to write down {\em all} bimolecular CRNs which are of minimal size in the above sense and admit Hopf bifurcation. We are also able to find which of these networks admit bifurcations of higher codimension, and prove the existence of more than one periodic orbit in several of these networks.

Some of the techniques and symbolic computations we use to analyse small networks in this paper rapidly grow in complexity with network size. However, results on small networks can naturally be combined with inheritance results, such as those gathered and developed in \cite{banajisplitreacs}, which tell us under what circumstances dynamical behaviours including oscillation in a CRN can be inferred from an analysis of its subnetworks. In the concluding section we illustrate how the results of this paper can be combined with inheritance results to make predictions about Hopf bifurcation in larger networks.

\subsection{Outline of the results}
\label{secoutline}

We first present the main results informally, with more precise statements and definitions to follow later. 

A key feature of any reaction network is its {\em rank}, namely, the dimension of the linear subspace spanned by its reaction vectors. Given a CRN of rank $r$, all nontrivial dynamics (if any), occurs on invariant sets of dimension at most $r$. For this reason, it is trivial that any autonomous ODE model of a CRN satisfying conditions for uniqueness of solutions, and with rank less than $2$, cannot have a nonconstant periodic orbit, and hence is incapable of nondegenerate Hopf bifurcation. 

Our interest here is in CRNs with mass action kinetics and which admit nondegenerate oscillation, namely periodic orbits which have exactly one Floquet multiplier equal to $1$ relative to their stoichiometric class (see \cite[Section~2]{bbhAMCrank} for more detail). There are many examples of rank-$2$ mass action CRNs admitting nondegenerate oscillation \cite{frank-kamenetsky:salnikov:1943, schnakenberg:1979, escher:1981, csaszar:jicsinszky:turanyi:1982, boros:hofbauer:2021a, boros:hofbauer:2022a}; however if we restrict attention to {\em bimolecular} CRNs, which are often regarded as physically more realistic \cite{erdi:toth:1989}, then the possibilities are more limited. 

While rank-$2$, bimolecular, mass action CRNs can admit periodic orbits, with well-known examples such as the Lotka reactions (see Remark~\ref{remdegenosci}) and the Ivanova reactions, from the proof of Theorem~6 in \cite{boros:hofbauer:2022a} we can conclude that Hopf bifurcation is ruled out in these systems (see also \cite{pota:1983,pota:1985}).

As a result, the smallest bimolecular mass action networks admitting Hopf bifurcation must have at least three species and rank at least three. They must also have at least four (irreversible) reactions by easy arguments to follow later (see Lemma~\ref{lemmin4reac}). Consequently, the study of Hopf bifurcation in bimolecular mass action CRNs begins with networks involving $3$ species and $4$ irreversible reactions and having rank $3$. We will refer to such networks as $(3,4,3)$ networks.

At this point we remark that Wilhelm and Heinrich \cite{WilhelmHeinrich1995,WilhelmHeinrich1996} used a somewhat different notion of ``smallest'' when studying a small bimolecular CRN admitting Hopf bifurcation: they gave minimality of the number of quadratic terms in the mass action ODEs higher importance than minimality of the number of reactions. The CRN they studied included five chemical reactions on three chemical species and consequently does not fall amongst those we focus on here; however the resulting differential equations included only one quadratic term which is, indeed, fewer than in any bimolecular $(3,4,3)$ networks admitting Hopf bifurcation that we study here.

In later work dating to 2009, Wilhelm \cite{Wilhelm2009} adopted a definition of ``smallest'' close to the one used here. In that paper he described a bimolecular $(3,4,3)$ CRN admitting Hopf bifurcation with mass action kinetics (see the discussion section in \cite{Wilhelm2009} and also \cite[Section 5]{boros:hofbauer:2022b} for further analysis of this network). Wilhelm's example demonstrated that the set of bimolecular $(3,4,3)$ mass action networks admitting Hopf bifurcation is nonempty. However, there are, up to isomorphism, 14670 bimolecular $(3,4,3)$ networks which admit positive equilibria. The question we aim to answer is: how many of these networks admit Hopf bifurcation with mass action kinetics?

We will show that Wilhelm's example is far from being unique. In fact, exactly 136 bimolecular $(3,4,3)$ networks admit nondegenerate Hopf bifurcation on the positive orthant, and these fall into 86 distinct equivalence classes, in a sense to be made precise later. Out of 86 networks which represent these 86 classes, we find that $57$ admit a supercritical Hopf bifurcation; $54$ admit a subcritical Hopf bifurcation, and $25$ admit both. These $25$ networks also admit a so-called Bautin bifurcation, which guarantees the existence of two periodic orbits at some values of the rate constants. In total, a stable equilibrium and a stable periodic orbit can coexist in $29$ of the networks, while an unstable equilibrium and an unstable periodic orbit can coexist in $2$ of the networks. These claims are the content of Theorems~\ref{thms3r4Hopfpotential},~\ref{thms3r4Hopf}~and~\ref{thmBautin} below. 

There is also (up to equivalence) a single, exceptional, bimolecular $(3,4,3)$ network which does not admit nondegenerate Hopf bifurcation, but robustly admits a degenerate Hopf bifurcation. This remarkable network is discussed briefly in the concluding section and explored further in \cite{BBHverticalHopf}.

\section{Preliminaries}

We collect basic notation and definitions needed later. 

Points and sets in $\mathbb{R}^n$ are referred to as {\bf positive} if they lie in the {\bf positive orthant} $\mathbb{R}^n_{+} := \{x \in \mathbb{R}^n\colon x_i > 0,\,\,i=1,\ldots,n\}$, and {\bf nonnegative} if they lie in the {\bf nonnegative orthant} $\mathbb{R}^n_{\geq 0}:= \{x \in \mathbb{R}^n\colon x_i \geq 0,\,\,i=1,\ldots,n\}$. 

We denote a {\bf vector of ones}, whose length is inferred from the context, by $\mathbf{1}$. 

Given a nonnegative integer vector $a = (a_1,\ldots, a_n)$, we adopt the standard convention that $x^a$ is an abbreviation for the {\bf monomial} $x_1^{a_1}x_2^{a_2}\cdots x_n^{a_n}$, while if $A$ is an $m \times n$ nonnegative integer matrix with rows $A_1, \ldots, A_m$, then $x^A$ denotes the {\bf vector of monomials} $(x^{A_1}, x^{A_2}, \ldots, x^{A_m})^{\mathrm{t}}$. The notation clearly extends to the case where $A$ is any real matrix, and in this case we obtain a vector of generalised monomials. 

\subsection{Matrices}
We need several notions from linear algebra and the theory of matrices. 

We will denote by $a \circ b$ the {\bf entrywise product} of two matrices or vectors $a$ and $b$ of the same dimensions. It is also convenient to denote by $a/b$ the {\bf entrywise quotient} (provided that no entry of $b$ is zero). We assume that ordinary matrix multiplication takes precedence over entrywise operations and so, for example, an expression such as ``$k\circ AB$'' is to be interpreted as ``$k\circ (AB)$''. 

We apply the {\bf logarithm} and {\bf square root} to positive vectors with the understanding that the function is applied to each entry.

If $v \in \mathbb{R}^n$, we denote by $\Delta_v$ the diagonal matrix whose $(i,i)$th entry is $v_i$. 

Given an $n \times m$ matrix $M$, and nonempty sets $\alpha \subseteq \{1, \ldots, n\}$, $\beta \subseteq \{1, \ldots, m\}$ we write $M(\alpha|\beta)$ for the submatrix of $M$ with rows from $\alpha$ and columns from $\beta$. If $|\alpha|=|\beta|$ we write $M[\alpha|\beta]$ for the {\bf minor} $\mathrm{det}\,M(\alpha|\beta)$. The minor $M[\alpha|\alpha]$ is a {\bf principal minor} of $M$.

An $n \times n$ matrix $M$ is {\bf sign-symmetric} \cite{hershkowitz} if oppositely placed minors cannot have opposite signs, namely, $M[\alpha|\beta]\,M[\beta|\alpha] \geq 0$ whenever $\alpha,\,\beta \subseteq \{1, \ldots, n\}$ and $|\alpha| = |\beta| \neq 0$.

The {\bf Cauchy--Binet formula} \cite{gantmacher} tells us how to compute minors of a product of matrices. Given an $n \times k$ matrix $A$, a $k \times m$ matrix $B$, and $\alpha \subseteq \{1, \ldots, n\}$ and $\beta \subseteq \{1, \ldots, m\}$ satisfying $|\alpha| = |\beta| \neq 0$, we have:
\[
(AB)[\alpha|\beta] = \sum_{\gamma} A[\alpha|\gamma]B[\gamma|\beta]
\] 
where $\gamma$ ranges over all subsets of $\{1, \ldots, k\}$ of size $|\alpha|$.

A {\bf $\bm{P_0}$ matrix} is a square matrix all of whose principal minors are nonnegative. A $P_0$ matrix cannot have negative real eigenvalues \cite{kellogg}.

We refer to square matrices whose eigenvalues all have negative real parts as {\bf Hurwitz matrices}. The closure of the set of $n \times n$ Hurwitz matrices consists of matrices whose eigenvalues all have non-positive real parts. 

In this paper, we refer to an $n \times n$ matrix $M$, not necessarily symmetric, as {\bf negative semidefinite} if $x^\mathrm{t}Mx\leq 0$ for all $x \in \mathbb{R}^n$. The matrix $M$ is negative semidefinite if and only if $-(M+M^\mathrm{t})$ is a $P_0$ matrix. It is easily seen that negative semidefinite matrices lie in the closure of the Hurwitz matrices.  

Given a square matrix $M$, we denote by $M^{[2]}$ its {\bf second additive compound} \cite{muldowney}. An explicit formula for the additive compound matrix can be found in \cite{li_wang}. In the case of greatest interest here, namely for a $3 \times 3$ matrix $M = (m_{ij})$, we have
\[
M^{[2]} = \left(\begin{array}{ccc}m_{11}+m_{22}& m_{23} &-m_{13}\\m_{32}&m_{11}+m_{33} & m_{12}\\-m_{31}&m_{21}&m_{22}+m_{33}\end{array}\right)\,.
\]
Additive compound matrices play an important role in the study of Hopf bifurcation \cite{Guckenheimer1997aa,abphopf}. The key property of second additive compound matrices used here is that the eigenvalues of $M^{[2]}$ are the sums of pairs of eigenvalues of $M$, counted with multiplicity. In particular, if $M$ has a pair of purely imaginary eigenvalues, then $M^{[2]}$ is singular.

\subsection{Chemical reaction networks}

We summarise briefly the key notions associated with CRNs needed here, and refer the reader to \cite{banajiCRNosci,bbhAMCrank} for more detail. 

Given chemical species $\mathsf{X}_1, \ldots, \mathsf{X}_n$, a {\bf complex} is a formal sum of the form $\sum a_i \mathsf{X}_i$, where we assume here that the coefficients $a_i$ are nonnegative integers. The coefficient $a_i$ in such a complex is termed the {\bf stoichiometric coefficient} of $\mathsf{X}_i$ in that complex. A complex $\sum a_i \mathsf{X}_i$ is termed {\bf bimolecular} if $\sum a_i \leq 2$. We remark that it would seem more accurate to refer to such complexes as ``at most bimolecular''; however, for brevity we refer to them simply as ``bimolecular''. Note that such complexes are termed ``short complexes'' in \cite{Horn_3_short}.

An (irreversible) chemical reaction involves the interconversion of one complex, termed the {\bf reactant complex}, into another, termed the {\bf product complex}. In this paper, a reaction will always mean an irreversible reaction. A reaction is termed bimolecular if its reactant and product complexes are both bimolecular. 

A CRN is simply a set of chemical reactions on some set of species. It is termed bimolecular if all its reactions are bimolecular. 

For the remainder of this section we consider an arbitrary CRN involving $n$ chemical species and $m$ reactions. We assume a fixed, but arbitrary ordering on species and reactions. Associated with the CRN is an $n \times m$ {\bf stoichiometric matrix}, whose $(i,j)$th entry tells us the net production of the $i$th species in the $j$th reaction. Each column of the stoichiometric matrix is termed a {\bf reaction vector}. The {\bf stoichiometric subspace} is the span of the reaction vectors, and the {\bf rank} of the CRN is the rank of its stoichiometric matrix, i.e., the dimension of the stoichiometric subspace. 

We will refer to a CRN with $n$ species, $m$ reactions, and rank $r$ as an $\bm{(n,m,r)}$ {\bf CRN}. Note that $r \leq \min\{n,m\}$. 

We define the {\bf left stoichiometric matrix} of a CRN to be the nonnegative integer matrix whose $(i,j)$th element is the stoichiometric coefficient of the $i$th species in the reactant complex of the $j$th reaction. 

In this paper we are concerned with CRNs with {\bf mass action kinetics}, which we will refer to as ``mass action CRNs''. A mass action CRN gives rise to a polynomial ODE on $\mathbb{R}^n$ which can be written compactly, using the notation developed above, as $\dot x = \Gamma (\kappa \circ x^{\Gamma_l^\mathrm{t}})$. Here $x$ is the vector of species concentrations, $\kappa$ is a positive vector of {\bf rate constants}, $\Gamma$ is the stoichiometric matrix, and $\Gamma_l$ is the left stoichiometric matrix. The vector $\kappa \circ x^{\Gamma_l^\mathrm{t}}$ is termed the {\bf rate vector} of the CRN. A bimolecular mass action CRN gives rise to an ODE system which is at most quadratic. In practice, we are interested in the evolution of species concentrations on $\mathbb{R}^n_{\geq 0}$. 

Two CRNs are termed {\bf isomorphic} if there is a relabelling/reordering of the species and/or reactions of one which gives us the other. 

Two CRNs are termed {\bf dynamically equivalent} if, perhaps after relabelling of species, they give rise to the same set of mass action differential equations (a precise definition is given in Appendix~\ref{appdyniso}). Such CRNs are termed ``unconditionally confoundable'' in \cite{CraciunPanteaIdentifiability}, where a criterion for such equivalence is given. Non-isomorphic CRNs may be dynamically equivalent as we see in Example~\ref{exdyneq} below.

\begin{example}[Dynamically equivalent, but non-isomorphic, CRNs]
\label{exdyneq}
Consider the pair of networks
\[
\mathsf{0} \rightarrow \mathsf{X},\quad 2\mathsf{X} \rightarrow 2\mathsf{Y}, \quad \mathsf{Y} \rightarrow 0
\]
and
\[
\mathsf{0} \rightarrow \mathsf{X},\quad 2\mathsf{X} \rightarrow \mathsf{X} + \mathsf{Y}, \quad \mathsf{Y} \rightarrow 0\,.
\]
It is easily seen that these CRNs are not isomorphic; they do, however, give rise to the same set of mass action differential equations and are thus dynamically equivalent in the sense used here. In fact they are equivalent in an even simpler sense described in Appendix~\ref{appdyniso}. 
\end{example}

\begin{remark}[More subtle equivalences between networks]
\label{remnondeq}
Two networks may fail to be dynamically equivalent in the sense used here, but nevertheless give rise to the same set of differential equations after some natural recoordinatisation. More details and an example are in Appendix~\ref{secnondeq}.
\end{remark}

We refer to a CRN with stoichiometric matrix $\Gamma \in \mathbb{R}^{n \times m}$ as {\bf dynamically nontrivial} if $\mathrm{ker}\,\Gamma \cap \mathbb{R}^m_{+}$ is nonempty, and dynamically trivial otherwise. It is easily shown that dynamically trivial CRNs can have no limit sets intersecting the positive orthant \cite{banajiCRNcount} for mass action kinetics or, indeed, under much weaker assumptions on the kinetics. Thus whenever our goal is to find CRNs with interesting positive limit sets, we can restrict attention to dynamically nontrivial CRNs. 

\begin{lemma}
\label{lemmin4reac}
Any bimolecular, $3$-species, mass action CRN with the capacity for Hopf bifurcation must have at least $4$ reactions. 
\end{lemma}
\begin{proof}
First, we can rule out Hopf bifurcation on the boundary of $\mathbb{R}^3_{\geq 0}$ in a $3$-species bimolecular mass action network. By elementary arguments, 
any boundary equilibrium must occur on an invariant face of $\mathbb{R}^3_{\geq 0}$; however, restricted to such a face, the network takes the form of a bimolecular CRN of rank less than $3$, and Hopf bifurcation is ruled out by Theorem~6 in \cite{boros:hofbauer:2022a}. So Hopf bifurcation in a $3$-species bimolecular mass action network must occur on the positive orthant. 

We have already observed that by Theorem~6 in \cite{boros:hofbauer:2022a} any bimolecular CRN capable of Hopf bifurcation must have rank at least 3 and hence at least 3 reactions. On the other hand any $(3,3,3)$ CRN (indeed, any CRN with rank equal to the number of reactions) has stoichiometric matrix with trivial kernel, and hence is dynamically trivial; this rules out positive equilibria, and Hopf bifurcation on the positive orthant. We arrive at the conclusion that a $3$-species, mass action CRN allowing Hopf bifurcation can have no fewer than four reactions.
\end{proof}

Using the algorithmic approaches described in \cite{banajiCRNcount}, we can enumerate bimolecular $(n,m,r)$ CRNs with specified properties for sufficiently small $n$ and $m$. From such enumeration, we obtain:
\begin{lemma}
\label{lems3r4}
Up to isomorphism, there are $14670$ dynamically nontrivial, bimolecular $(3,4,3)$ networks. These fall into $9259$ dynamically non-equivalent classes.
\end{lemma}
Code and scripts to carry out the enumeration of CRNs, and other analysis to follow, are available on GitHub \cite{muradgithub, balazsgithub}. 

Ruling out Hopf bifurcation in the great majority of the $(3,4,3)$ networks in Lemma~\ref{lems3r4} becomes straightforward once we have developed a little technical machinery.

\subsection{Equilibria and Jacobian matrices of mass action networks}

Recall that a mass action CRN gives rise to an ODE of the form:
\begin{equation}
\label{eqMA}
\dot x = \Gamma (\kappa \circ x^{\Gamma_l^\mathrm{t}})\,.
\end{equation}
We assume that the CRN giving rise to this ODE is a dynamically nontrivial $(n,m,r)$ CRN, so that $\mathrm{ker}_{+}\,\Gamma := \mathrm{ker}\,\Gamma \cap \mathbb{R}^m_{+}$ is nonempty. For any fixed $\kappa$, a positive equilibrium $x_0$ satisfies $\kappa \circ x_0^{\Gamma_l^\mathrm{t}} \in \mathrm{ker}_{+}\,\Gamma$. On the other hand, given any $y \in \mathrm{ker}_{+}\,\Gamma$, and any positive $x_0$, setting
\[
\kappa  = y \circ x_0^{-\Gamma_l^\mathrm{t}}
\]
ensures that $\kappa \circ x_0^{\Gamma_l^\mathrm{t}} = y$, i.e., with this choice of rate constants, $x_0$ is an equilibrium of the system and the rate vector at $x_0$ is precisely $y$. 

Let $K$ be the closure of $\mathrm{ker}_{+}\,\Gamma$. The set $K$ is an $(m-r)$-dimensional, pointed, polyhedral cone in $\mathbb{R}^m$ with relative interior $\mathrm{ker}_{+}\,\Gamma$. Let $K$ have, up to positive scaling, $k$ extreme vectors, say $u_1, \ldots, u_k$, so that we can write any element of $\mathrm{ker}_{+}\,\Gamma$ as $\sum_{j=1}^k \mu_j u_j$ where $\mu_j > 0$. Then $\mathrm{ker}_{+}\,\Gamma$ is parameterised by the parameter $\mu := (\mu_1, \ldots, \mu_k)^\mathrm{t} \in \mathbb{R}^k_{+}$. Observe that if $\mathrm{dim}\,K =m-r \leq 2$, then in fact $k= m-r$, and so the $\mu_j$ are {\em uniquely} defined for each point in $\mathrm{ker}_{+}\,\Gamma$; but if $m-r \geq 3$, then it is possible that $k>m-r$, in which case the $\mu_j$ are not uniquely defined.

On the positive orthant the ODE (\ref{eqMA}) has Jacobian matrix
\begin{equation}
\label{eqJacMA}
\Gamma \Delta_{v(x,\kappa)} \Gamma_l^{\mathrm{t}}\Delta_{1/x}\,,
\end{equation}
where $\Delta_{v(x,\kappa)}$ is a diagonal matrix whose diagonal entries are the entries of $v(x,\kappa):=\kappa \circ x^{\Gamma_l^\mathrm{t}}$, and $\Delta_{1/x}$ is a diagonal matrix whose diagonal entries are $1/x_i$. 

We say that an equilibrium of a CRN is {\bf nondegenerate} (resp., {\bf degenerate}), if it is nondegenerate (resp., degenerate) relative to the stoichiometric subspace of the CRN (see Section~2.2 in \cite{banajipantea} for a variety of equivalent formulations of this condition). In the mass action case, a positive equilibrium, say $x_0$, of (\ref{eqMA}) with rate constants $\kappa_0$ is nondegenerate if and only if $\mathrm{rank}\,(\Gamma \Delta_{v(\kappa_0,x_0)} \Gamma_l^{\mathrm{t}}\Delta_{1/x_0}\Gamma) = \mathrm{rank}\,\Gamma$, i.e., the Jacobian matrix evaluated at the equilibrium acts as a nonsingular transformation on the stoichiometric subspace. If $\mathrm{rank}\,\Gamma$ is equal to the number of species, then this condition for nondegeneracy of $x_0$ becomes simply that the Jacobian matrix $\Gamma \Delta_{v(\kappa_0,x_0)} \Gamma_l^{\mathrm{t}}\Delta_{1/x_0}$ has maximal rank. 

Let us now assume that the extreme vectors $u_i$ of $K$ have been chosen and fixed. Using (\ref{eqJacMA}) and the previous remarks, the set of Jacobian matrices of the network {\em at positive equilibria} can be written
\begin{equation}
\label{eqJacMA1}
\mathcal{J}:= \{\Gamma \Delta_{\sum \mu_j u_j} \Gamma_l^{\mathrm{t}}\Delta_{z}\colon z  \in \mathbb{R}^n_{+},\,\,\mu \in \mathbb{R}^k_{+}\}
\end{equation}
where we have set $z_i = 1/x_i$, and $\Delta_{\sum \mu_j u_j}$ is the diagonal matrix whose $(i,i)$th entry is the $i$th entry of $\sum_{j=1}^k \mu_j u_j$. Note that each diagonal entry of $\Delta_{\sum \mu_j u_j}$ is a homogeneous linear form in the variables $\mu = (\mu_1, \ldots, \mu_k)$. It is straightforward that each matrix in $\mathcal{J}$ is indeed the Jacobian matrix of (\ref{eqMA}) at some equilibrium for some choice of rate constants: given any $z\in\mathbb{R}^n_{+},\,\,\mu \in \mathbb{R}^k_{+}$, set $x = 1/z$ and $\kappa  = \left(\sum \mu_j u_j\right)\circ x^{-\Gamma_l^\mathrm{t}}$; then $\kappa \circ x^{\Gamma_l^\mathrm{t}} = \sum \mu_j u_j$, and $x$ is thus an equilibrium of the system with Jacobian matrix $\Gamma \Delta_{\sum \mu_j u_j} \Gamma_l^{\mathrm{t}}\Delta_{z}$. The set of Jacobian matrices of (\ref{eqMA}) at equilibria is thus parameterised by $n+k$ positive parameters $z_1, \ldots, z_n$ and $\mu_1, \ldots, \mu_k$. 

\begin{remark}[Positive homogeneity of the set of Jacobian matrices]
\label{remJhom}
Note that the family of Jacobian matrices in (\ref{eqJacMA1}) is overparameterised, being homogeneous (under positive scaling) in both $z$ and $\mu$. I.e., setting $J(z, \mu) := \Gamma \Delta_{\sum \mu_j u_j} \Gamma_l^{\mathrm{t}}\Delta_{z}$, and choosing $s,t$ to be any positive real numbers, we have $J(sz, t\mu) = stJ(z, \mu)$. Thus we can always reduce the number of parameters in the parameterisation of the Jacobian matrices by $1$; and if we are interested only in properties of $J$ which are invariant under scaling, then we can reduce it by $2$. For example, in the case where $n=3$ and $\mathrm{dim}\,(\mathrm{ker}\,\Gamma)=1$, we have, up to positive scaling, only a two-parameter family of Jacobian matrices at equilibria. 
\end{remark}

\section{Dynamically nontrivial $(n,n+1,n)$ networks}
\label{secDN}

Let us now consider CRNs in the class of most interest here: dynamically nontrivial $(n,n+1,n)$ networks. Analysis of these networks is considerably simpler than the general case. Unless explicitly stated otherwise, the results in Sections~\ref{secDN}~and~\ref{secHopfgeneral} do not assume bimolecularity of the CRNs involved.

First, we observe that checking dynamical equivalence is straightforward for these networks. This is because amongst dynamically nontrivial $(n,n+1,n)$ CRNs for any fixed $n$, dynamical equivalence reduces to a simpler, easily checked, condition. The claim about dynamical equivalence in Lemma~\ref{lems3r4} takes advantage of this simplification. The details are in Appendix~\ref{appdyniso}. 

Second, the parameterisation of the set of Jacobian matrices at equilibria given in (\ref{eqJacMA1}) simplifies for dynamically nontrivial $(n,n+1,n)$ networks because $\mathrm{ker}\,\Gamma$ is one dimensional and includes a strictly positive vector. Let $u$ be some positive vector in $\mathrm{ker}\,\Gamma$, so that $\mathrm{ker}_{+}\,\Gamma = \{\mu \,u\colon \mu > 0\}$. Each Jacobian matrix appearing in (\ref{eqJacMA1}) now has the form $\mu\,Q\,\Delta_{z}$, where $Q:=\Gamma \Delta_{u} \Gamma_l^{\mathrm{t}}$ is a real $n \times n$ matrix. The set of Jacobian matrices evaluated at equilibria is then simply
\begin{equation}
\label{eqJacMA2}
\mathcal{J} = \{Q\,\Delta_{z}\,:\, z \in \mathbb{R}^n_{+}\}\,.
\end{equation}
Note that the parameter $\mu$ has disappeared altogether, which can be seen as a consequence of Remark~\ref{remJhom}. The Jacobian matrices thus form an $n$-parameter family of homogeneous polynomial matrices with a constant first factor and a second factor which is a diagonal matrix of positive unknowns. As we shall see, several properties of the matrices in $\mathcal{J}$ can be inferred from the constant matrix $Q$ alone.

\subsection{Degenerate and nondegenerate $\bm{(n,n+1,n)}$ networks}
We next prove that dynamically nontrivial, mass action, $(n, n+1, n)$ networks fall into two categories: ``nondegenerate'' networks which admit a single nondegenerate positive equilibrium for all values of rate constants; and ``degenerate'' networks which admit no positive equilibria for some values of rate constants and a continuum of degenerate positive equilibria for other values of the rate constants. 

Given an $(n,n+1,n)$ CRN having left stoichiometric matrix $\Gamma_l$, define the condition 
\begin{equation}
\tag{ND}
\mathrm{rank}\,[\Gamma_l^{\mathrm{t}}\,|\,-\mathbf{1}] = n+1.
\end{equation}
We define a dynamically nontrivial, $(n,n+1,n)$, mass action CRN to be nondegenerate if it satisfies condition (ND), and to be degenerate otherwise. This is justified by the following lemma.

\begin{lemma}
\label{lemnondegen}
Let $\mathcal{R}$ be a dynamically nontrivial, $(n,n+1,n)$, mass action CRN. Then
\begin{enumerate}[align=left,leftmargin=*]
\item If $\mathcal{R}$ satisfies condition (ND) then it admits a single positive equilibrium for all positive rate constants, and this equilibrium is nondegenerate.
\item If $\mathcal{R}$ fails condition (ND) then it admits no positive equilibria for almost all positive rate constants and a continuum of positive equilibria, all of which are degenerate, for an exceptional set of rate constants. 
\end{enumerate}
\end{lemma}

\begin{proof}
Let $\mathcal{R}$ have stoichiometric matrix $\Gamma$ and left stoichiometric matrix $\Gamma_l$. Recall that for given rate constants $\kappa \in \mathbb{R}^{n+1}_{+}$, the point $x \in \mathbb{R}^n_{+}$ is an equilibrium of $\mathcal{R}$ if and only if there exists $\mu>0$ such that
\begin{equation}
\label{eqeq}
\kappa \circ x^{\Gamma_l^\mathrm{t}} = \mu u\,,
\end{equation}
where $u$ is some arbitrary but fixed positive element in $\mathrm{ker}\,\Gamma$. Taking logarithms and rearranging, (\ref{eqeq}) is satisfied if and only if there exist $x \in \mathbb{R}^n_{+},\,\,\mu > 0$ such that $\Gamma_l^{\mathrm{t}}\ln\,x = \ln(\mu)\mathbf{1} + \ln\,u - \ln\,\kappa$, namely,
\begin{equation}
\label{eqMAdegen}
[\Gamma_l^{\mathrm{t}}\,|\,-\mathbf{1}]\left(\begin{array}{c}\ln\,x\\\ln\,\mu\end{array}\right) = \ln\,u - \ln\,\kappa\,.
\end{equation}

Let us now show that the condition (ND) is equivalent to nondegeneracy of all positive equilibria. Recall that each Jacobian matrix at a positive equilibrium has the form $J= Q \Delta_{z}$ where $Q:= \Gamma \Delta_{u} \Gamma_l^\mathrm{t}$ and $z \in \mathbb{R}^n_{+}$. Clearly, nonsingularity of $J$ does not depend on $z$: i.e., $J$ is nonsingular if and only if $Q$ is nonsingular. This, in turn, is equivalent to (ND). To see this we note that $\mathrm{rank}\,[\Gamma_l^{\mathrm{t}}\,|\,-\mathbf{1}] < n+1$ implies either that $\mathrm{rank}\,\Gamma_l^{\mathrm{t}} < n$ which immediately implies that $Q$ is singular; or that $\mathbf{1} \in \mathrm{im}\,\Gamma_l^{\mathrm{t}}$, in which case $Q$ is again singular because $\Delta_{u}\Gamma_l^\mathrm{t}$ maps some nonzero vector $v$ to $u$, which in turn lies in $\mathrm{ker}\,\Gamma$; and thus $v \in \mathrm{ker}\,Q$. On the other hand, if $\mathrm{rank}\,[\Gamma_l^{\mathrm{t}}\,|\,-\mathbf{1}] = n+1$, then (i) $\mathrm{rank}\,\Delta_{u}\Gamma_l^\mathrm{t} = n$, and (ii) $\mathrm{im}\,(\Delta_{u}\Gamma_l^\mathrm{t}) \cap \mathrm{ker}\,\Gamma = \{0\}$; and so $Q$ (and hence $J$) must be nonsingular. 

We now show that nondegenerate and degenerate networks indeed have the properties claimed.
\begin{enumerate}[align=left,leftmargin=*]
\item {\bf Nondegenerate networks.} Suppose (ND) holds, so that $[\Gamma_l^{\mathrm{t}}\,|\,-\mathbf{1}]$ has rank $n+1$ and is hence invertible. For each $\kappa \in \mathbb{R}^{n+1}_{+}$, we can then solve (\ref{eqMAdegen}) to find $\ln\,x$ and $\ln\,\mu$ (and hence positive $x$ and $\mu$) satisfying (\ref{eqMAdegen}):
\begin{equation}
\label{eqxmukappa}
\left(\begin{array}{c}\ln\,x\\\ln\,\mu\end{array}\right) = [\Gamma_l^{\mathrm{t}}\,|\,-\mathbf{1}]^{-1}(\ln\,u - \ln\,\kappa)\,, \quad \mbox{namely} \quad \left(\begin{array}{c}x\\\mu\end{array}\right) = (u/\kappa)^{[\Gamma_l^{\mathrm{t}}\,|\,-\mathbf{1}]^{-1}}\,.
\end{equation}
Since the solution is unique, the system has a unique positive equilibrium for any positive $\kappa$. From the discussion of the Jacobian matrices above, this equilbrium is nondegenerate. 

\item {\bf Degenerate networks.} Suppose (ND) fails, so that $\mathrm{rank}\,[\Gamma_l^{\mathrm{t}}\,|\,-\mathbf{1}] < n+1$. If we choose $\kappa$ such that $\ln\,u - \ln\,\kappa \not\in \mathrm{im}\,[\Gamma_l^{\mathrm{t}}\,|\,-\mathbf{1}]$, which is the case for almost all $\kappa$ as $\ln\,\kappa$ varies over all of $\mathbb{R}^{n+1}$, then (\ref{eqMAdegen}) cannot be satisfied and the network has no positive equilibria. On the other hand, consider any $\kappa$ such that $\ln\,u - \ln\,\kappa \in \mathrm{im}\,[\Gamma_l^{\mathrm{t}}\,|\,-\mathbf{1}]$, i.e., such that there exist $(x,\mu)$ satisfying (\ref{eqMAdegen}). Then we can add to $(\ln\,x, \ln\,\mu)^\mathrm{t}$ any element of $\mathrm{ker}\,[\Gamma_l^{\mathrm{t}}\,|\,-\mathbf{1}]$ and again satisfy (\ref{eqMAdegen}) with the same rate constants $\kappa$. Since any nonzero element of $\mathrm{ker}\,[\Gamma_l^{\mathrm{t}}\,|\,-\mathbf{1}]$ must have at least one nonzero entry in its first $n$ entries, there is clearly a continuum of positive equilibria of the system for such values of $\kappa$. From the discussion of Jacobian matrices above, every such equilibrium must be degenerate.
\end{enumerate}
This concludes the proof.
\end{proof}

\begin{remark}[The condition for nondegeneracy]
\label{remND}
We observe that condition (ND) is equivalent to affine independence of the reactant complexes of an $(n,n+1,n)$ CRN. We also observe from the proof of Lemma~\ref{lemnondegen} that for a dynamically nontrivial $(n,n+1,n)$ network condition (ND) is equivalent to nonsingularity of the matrix $Q = \Gamma \Delta_{u} \Gamma_l^\mathrm{t}$, where $u$ is any element of $\mathrm{ker}_{+}\Gamma$.
\end{remark}

\begin{remark}[Degenerate CRNs]
\label{remdegen}
Going beyond $(n, n+1, n)$ networks, it is natural to refer to any dynamically nontrivial, mass action CRN all of whose positive equilibria are degenerate as a degenerate network. Clearly a sufficient (but not necessary) condition for a dynamically nontrivial CRN to be degenerate is if $\mathrm{rank}\,\Gamma_l < \mathrm{rank}\,\Gamma$.
\end{remark}

It is algorithmically straightforward to check condition (ND) and find out if an $(n, n+1, n)$ network is degenerate. Consequently, we find the following:
\begin{lemma}
\label{lems3r4nondegen}
Out of the $14670$ dynamically nontrivial, bimolecular $(3,4,3)$ CRNs (Lemma~\ref{lems3r4}), exactly $10853$ are nondegenerate. These fall into $6486$ dynamically non-equivalent classes.
\end{lemma}

We illustrate some of the computations described above via an example.
\begin{example}[Basic analysis of a network]
\label{exnondegen}
Consider the following CRN:
\[
\mathsf{X} \rightarrow 2\mathsf{X}, \quad \mathsf{X}+\mathsf{Y} \rightarrow 2\mathsf{Z}, \quad \mathsf{X}+\mathsf{Z} \rightarrow \mathsf{Y}, \quad \mathsf{Y}+\mathsf{Z} \rightarrow \mathsf{Y}\,.
\]
The stoichiometric matrix and left stoichiometric matrix of the network are
\[
\Gamma = \left(\begin{array}{rrrr}1&-1&-1&0\\0&-1&1&0\\0&2&-1&-1\end{array}\right) \quad \mbox{and} \quad \Gamma_l = \left(\begin{array}{cccc}1&1&1&0\\0&1&0&1\\0&0&1&1\end{array}\right)\,.
\]
As $\Gamma$ has rank $3$, this is a $(3,4,3)$ network. Choosing $u=(2,1,1,1)^\mathrm{t} \in \mathrm{ker}\,\Gamma$, we see that the network is dynamically nontrivial. The matrix 
\[
[\Gamma_l^\mathrm{t}\,|\,-\mathbf{1}] = \left(\begin{array}{rrrr}1&0&0&-1\\1&1&0&-1\\1&0&1&-1\\0&1&1&-1\end{array}\right)
\]
is easily checked to have rank $4$, confirming that the network is nondegenerate. With $u=(2,1,1,1)^\mathrm{t}$, we obtain $Q$ and the general expression for the Jacobian matrix in terms of the variables $z_i = 1/x_i$:
\[
Q = \Gamma \Delta_{u} \Gamma_l^\mathrm{t} = \left(\begin{array}{rrr}0&-1&-1\\0&-1&1\\1&1&-2\end{array}\right), \quad J=Q\,\Delta_z = \left(\begin{array}{ccc}0&-z_2&-z_3\\0&-z_2&z_3\\z_1&z_2&-2z_3\end{array}\right)\,.
\]
\end{example}

\subsection{A natural recoordinatisation}
\label{secrecoord}

From the proof of Lemma~\ref{lemnondegen}, we observe that in nondegenerate, dynamically nontrivial, $(n,n+1,n)$, mass action networks, we can easily pass between expressions for the Jacobian matrices at equilibria in terms of variables $x_i$ (or their inverses $z_i$), and in terms of the rate constants $\kappa$. Defining $G:=[\Gamma_l^{\mathrm{t}}\,|\,-\mathbf{1}]^{-1}$, we found that
\[
\left(\begin{array}{c}x\\\mu\end{array}\right) = (u/\kappa)^G\,, \quad \mbox{or, equivalently} \quad\left(\begin{array}{c}z\\\mu^{-1}\end{array}\right) = (\kappa/u)^G\,.
\]
Thus we can, if desired, substitute $z_i = ((\kappa/u)^{G})_i$, $i = 1, \ldots, n$ in the expression for a general Jacobian matrix of the system at an equilbrium, namely $J=Q\Delta_z$, to obtain this matrix in terms of the rate constants.

These observations lead to a natural recoordinatisation which brings the unique equilibrium to $\mathbf{1}$. The recoordinatised form considerably simplifies some of the analysis of bifurcations in these networks. Define $x^*, \mu^*$ via
\[
\left(\begin{array}{c}x^*\\\mu^*\end{array}\right) = (u/\kappa)^G\,,
\]
and define new variables $X = x/x^* \in \mathbb{R}^n_{+}$. Define $G' = G(\{1,\ldots, n\}|\{1, \ldots, n+1\})$ (i.e., $G'$ is the top $n$ rows of $G$), and note that $\Gamma_l^\mathrm{t}G' = I + \mathbf{1}v^\mathrm{t}$, where $I$ is the $(n+1) \times (n+1)$ identity matrix, and $v^\mathrm{t}$ is the last row of $G$. (The product $\mathbf{1}v^t$ is a rank one $(n+1) \times (n+1)$ matrix.) Set $G'' = G'- \mathbf{1}v^t$, i.e., subtract the last row of $G$ from each of the first $n$ rows of $G$ to get the $n \times (n+1)$ matrix $G''$. With these preliminaries we can now write the differential equation (\ref{eqMA}), namely $\dot x = \Gamma (\kappa \circ x^{\Gamma_l^\mathrm{t}})$, in terms of the new variables $X$:
\begin{eqnarray*}
\dot X = \frac{1}{x^*}\dot x &=& \frac{1}{x^*} \circ \Gamma (\kappa \circ (x^* \circ X)^{\Gamma_l^\mathrm{t}})\\
&=& (\kappa/u)^{G'}\circ \Gamma (\kappa \circ ((u/\kappa)^{G'})^{\Gamma_l^\mathrm{t}} \circ X^{\Gamma_l^\mathrm{t}})\\
&=& (\kappa/u)^{G'}\circ \Gamma (\kappa \circ (u/\kappa)^{\Gamma_l^\mathrm{t}G'} \circ X^{\Gamma_l^\mathrm{t}})\\
&=& (\kappa/u)^{G'}\circ \Gamma (\kappa \circ (u/\kappa) \circ (u/\kappa)^{\mathbf{1}v^\mathrm{t}} \circ X^{\Gamma_l^\mathrm{t}}) \quad \mbox{(using $\Gamma_l^\mathrm{t}G' = I + \mathbf{1}v^\mathrm{t}$)}\\
&=& (\kappa/u)^{G'}\circ \Gamma (u \circ ((u/\kappa)^{v^\mathrm{t}})^{\mathbf{1}} \circ X^{\Gamma_l^\mathrm{t}})\,\\
&=& (u/\kappa)^{v^\mathrm{t}}(\kappa/u)^{G'}\circ \Gamma (u \circ X^{\Gamma_l^\mathrm{t}})\\
&=& (\kappa/u)^{G''}\circ \Gamma (u \circ X^{\Gamma_l^\mathrm{t}})\,.
\end{eqnarray*}
Note that $(\kappa/u)^{G''}$ is a vector of homogeneous generalised monomials in $\kappa$. Note also that $\mathbf{1}$ is now clearly an equilibrium of the network, since $\Gamma u = 0$. Since $G''$ has rank $n$ and hence defines a surjective linear transformation, it easily follows that the map $\kappa \mapsto (\kappa/u)^{G''}$ is surjective when regarded as a map from $\mathbb{R}^{n+1}_{+}$ to $\mathbb{R}^n_{+}$. In other words, if we consider the elements of $(\kappa/u)^{G''}$ as new parameters, then these vary over $\mathbb{R}^n_{+}$.

\begin{example}[Illustrating the natural coordinate transformation]
We return to the nondegenerate, dynamically nontrivial $(3,4,3)$ network in Example~\ref{exnondegen}, namely
\[
\mathsf{X} \rightarrow 2\mathsf{X}, \quad \mathsf{X}+\mathsf{Y} \rightarrow 2\mathsf{Z}, \quad \mathsf{X}+\mathsf{Z} \rightarrow \mathsf{Y}, \quad \mathsf{Y}+\mathsf{Z} \rightarrow \mathsf{Y}\,.
\]
The network gives rise to the mass action ODE system
\[
\left(\begin{array}{c}\dot x\\\dot y \\\dot z\end{array}\right) = \left(\begin{array}{rrrr}1&-1&-1&0\\0&-1&1&0\\0&2&-1&-1\end{array}\right)\left(\begin{array}{c}\kappa_1x\\\kappa_2xy\\\kappa_3xz\\\kappa_4yz\end{array}\right)\,,
\]
where $x,y$ and $z$ are the concentrations of $\mathsf{X}$, $\mathsf{Y}$ and $\mathsf{Z}$ respectively. To obtain the recoordinatised system, we calculate
\[
G=[\Gamma_l^\mathrm{t}\,|\,-\mathbf{1}]^{-1} = \left(\begin{array}{rrrr}-1&1&1&-1\\-1&1&0&0\\-1&0&1&0\\-2&1&1&-1\end{array}\right)\, \quad \mbox{and} \quad G'' = \left(\begin{array}{rrrr}1&0&0&0\\1&0&-1&1\\1&-1&0&1\end{array}\right)\,.
\]
Choosing $u=(2,1,1,1)^\mathrm{t} \in \mathrm{ker}\,\Gamma$, the recoordinatised system takes the form
\[
\left(\begin{array}{c}\dot X\\\dot Y \\\dot Z\end{array}\right) = \left(\begin{array}{c}\kappa_1/2\\\kappa_1\kappa_4/(2\kappa_3)\\\kappa_1\kappa_4/(2\kappa_2)\end{array}\right) \circ \left(\begin{array}{rrrr}1&-1&-1&0\\0&-1&1&0\\0&2&-1&-1\end{array}\right)\,\left(\begin{array}{c}2X\\XY\\XZ\\YZ\end{array}\right)\,.
\]
where $X,Y$ and $Z$ are the rescaled concentrations of $\mathsf{X}$, $\mathsf{Y}$ and $\mathsf{Z}$ respectively.
\end{example}

\section{Conditions for Hopf bifurcation}
\label{secHopfgeneral}

We next turn to the task of writing down conditions which rule out Hopf bifurcation in the great majority of systems of interest to us here. We begin in some generality. 

\subsection{Necessary conditions for Hopf bifurcation}

Two necessary conditions for nondegenerate Hopf bifurcation in a family of ODEs are:
\begin{enumerate}[align=left,leftmargin=*]
\item[H1.] The family admits nondegenerate equilibria. 
\item[H2.] The Jacobian matrices of the family, evaluated at nondegenerate equilibria, admit nonreal eigenvalues with negative, zero and positive real parts.
\end{enumerate}

\begin{remark}
Conditions H1 and H2 above are clearly not sufficient for nondegenerate Hopf bifurcation: for this, we also need to confirm the nondegeneracy and transversality conditions associated with the bifurcation, discussed further in Section~\ref{sec343hopf}. But, to preview what we will show: amongst bimolecular, $(3,4,3)$, mass action CRNs, conditions H1 and H2 are almost sufficient to guarantee nondegenerate Hopf bifurcation. Remarkably, of the 6486 dynamically non-equivalent, nondegenerate, dynamically nontrivial, bimolecular $(3,4,3)$ networks (see Lemma~\ref{lems3r4nondegen}), there is precisely {\em one} which satisfies H1 and H2 but fails to admit nondegenerate Hopf bifurcation with mass action kinetics.
\end{remark}

We now write down conditions, termed N1--N4 below, which allow us to rule out H2. We first present the conditions in generality, before observing how they can simplify in the case of nondegenerate $(n,n+1,n)$ CRNs.

Consider a parameterised ODE system 
\begin{equation}
\label{eq0}
\dot x = f(x,\alpha)\,.
\end{equation}
We assume that the variables $x$ and the parameters $\alpha$ vary on $\mathbb{R}^n_{+}$ and $\mathbb{R}^k_{+}$ respectively, and that $f$ is $C^1$ on $\mathbb{R}^n_{+} \times \mathbb{R}^k_{+}$. Since we are thinking of $\alpha$ as rate constants, we also assume that $f(x, c\alpha) = c f(x,\alpha)$ for any $c>0$. Note that this condition implies that if $f(x,\alpha)=0$, then $f(x,c\alpha)=0$ for any positive $c$.

Let $D_xf(x,\alpha)$ be the Jacobian matrix of (\ref{eq0}), and note that we can regard $D_xf(\cdot,\cdot)$ as a continuous function on $\mathbb{R}^n_{+} \times \mathbb{R}^k_{+}$ with image in $\mathbb{R}^{n \times n}$. Let $E \subseteq \mathbb{R}^n_{+} \times \mathbb{R}^k_{+}$ be the set of positive equilibria of \eqref{eq0}, and write $J := \left. D_xf\right|_E$. For $J$ (or any other such function) we abuse terminology and say that $J$ belongs to $\mathcal{M} \subseteq \mathbb{R}^{n \times n}$ if $\mathrm{im}\,J \subseteq \mathcal{M}$. For example, we may say that ``$J$ is nonsingular'' to mean that all matrices in the image of $J$ are nonsingular. We may also write $\mathrm{det}\,J \neq 0$ to mean the same thing.

The following are computable conditions sufficient to guarantee that condition H2 cannot be satisfied, and thus to rule out Hopf bifurcation in (\ref{eq0}). The list is not exhaustive, but these are the conditions which prove most useful here. 
\begin{enumerate}[align=left,leftmargin=*]
\item[N1.] $J$ lies in the closure of the Hurwitz matrices, namely $J$ has no eigenvalues with positive real parts. In general, this may be hard to confirm; but here we find that $J$ frequently satisfies a simple condition which guarantees that this holds.

\item[N2.] $J^2$ is a $P_0$ matrix. As $P_0$ matrices have no negative real eigenvalues, and the eigenvalues of $J^2$ are the squares of those of $J$, this means that $J$ cannot have a nonzero imaginary eigenvalue.  

\item[N3.] $J^{[2]}$, the second additive compound matrix of $J$, is nonsingular; or $n\leq 3$ and $J^{[2]}$ is identically singular. Recall that if $J^{[2]}$ is nonsingular then no pair of eigenvalues can sum to zero and, in particular, $J$ cannot have a nonzero imaginary eigenvalue. If $n\leq 3$ and $J^{[2]}$ is identically singular, then this too is sufficient to rule out Hopf bifurcation; for in this case although a pair of imaginary eigenvalues might occur, the {\em passage} of a pair of nonreal eigenvalues through the imaginary axis is forbidden.

\item[N4.] $\mathrm{det}(J^2+I)$ is positive. Note that $\mathrm{det}(J^2+I)$ is the sum of the squares of the real and imaginary parts of $\mathrm{det}(J+iI)$ and so is automatically nonnegative. The condition that it is positive is thus precisely the condition that $J$ cannot have a pair of eigenvalues $\pm i$. This is sufficient to rule out Hopf bifurcation, in light of our assumption that $f$ (and hence $J$) is positively homogeneous in $\alpha$. 
\end{enumerate}

\begin{remark}
Condition N2 implies condition N4, but not vice versa; we list both conditions, because condition N2 can be considerably easier to check than condition N4. As we will see below, for the networks we treat here, N2 can be checked using arithmetic alone. 
\end{remark}

\begin{remark}
\label{remdegenosci}
A network may satisfy some of the conditions N1--N4, forbidding Hopf bifurcation, but nevertheless permit degenerate oscillation. The best-known example is the Lotka reactions $\mathsf{X} \rightarrow 2\mathsf{X}, \,\, \mathsf{X}+\mathsf{Y} \rightarrow 2\mathsf{Y}, \,\, \mathsf{Y} \rightarrow 0$: the mass action Jacobian matrix $J$ at the unique positive equilibrium lies in the closure of the Hurwitz matrices, and $J^{[2]}$ is identically singular; but for all values of the rate constants, the corresponding mass action system has a first integral and all nonconstant solutions are periodic. 
\end{remark}

\subsection{Conditions for Hopf bifurcation in $(n,n+1,n)$ networks}

Let us now turn to the special case of nondegenerate $(n,n+1,n)$ networks. In this case, some of the conditions N1--N4 ruling out Hopf bifurcation take a simpler form. 

Recall, from (\ref{eqJacMA2}), that the mass action Jacobian matrix at equilibria of a nondegenerate $(n,n+1,n)$ CRN takes the form $J=Q \Delta_z$ with $Q$ a constant matrix. In fact, we can always choose $Q$ to be an integer matrix. We have the following lemma.
\begin{lemma}
\label{lemspecnohopf}
Consider a nondegenerate $(n,n+1,n)$ CRN. Then
\begin{enumerate}
\item If $Q$ is negative semidefinite, then $J$ lies in the closure of the Hurwitz matrices, and hence nondegenerate Hopf bifurcation is ruled out.
\item The matrix $J^2$ is a $P_0$ matrix if and only if $Q$ is sign-symmetric. In this case, nondegenerate Hopf bifurcation is again ruled out.
\end{enumerate}
\end{lemma}
\begin{proof}
Let $J=Q\Delta_z$ where $z \in \mathbb{R}^n_{+}$.
\begin{enumerate}[align=left,leftmargin=*]
\item If $Q$ is negative semidefinite, then clearly $\Delta_{\sqrt{z}}\,Q\,\Delta_{\sqrt{z}}$ is negative semidefinite for all $z \in \mathbb{R}^n_{+}$, and hence $\Delta_{\sqrt{z}}\,Q\,\Delta_{\sqrt{z}}$ has no eigenvalues with positive real part. However 
\[
Q\Delta_z = \Delta_{1/\sqrt{z}}\,\Delta_{\sqrt{z}}\,Q\,\Delta_{\sqrt{z}}\,\Delta_{\sqrt{z}}
\]
is similar to $\Delta_{\sqrt{z}}\,Q\,\Delta_{\sqrt{z}}$; hence $Q\Delta_z$ can have no eigenvalue with positive real part.
\item Observe that $J[\alpha|\beta] = Q[\alpha|\beta]\,\Delta_z[\beta|\beta]$, as $\Delta_z$ is a diagonal matrix and hence all minors of the form $\Delta_z[\gamma|\beta]$ other than $\Delta_z[\beta|\beta]$ are zero. Hence, by the Cauchy--Binet formula,
\[
J^2[\alpha|\alpha] = \sum_\beta J[\alpha|\beta]J[\beta|\alpha] = \Delta_z[\alpha|\alpha]\sum_\beta Q[\alpha|\beta]\,\Delta_z[\beta|\beta]\,Q[\beta|\alpha]\,.
\]
It is now clear that if $J^2[\alpha|\alpha]$ fails to be nonnegative for some $\alpha$, then (since $\Delta_z[\beta|\beta]>0$ for all $\beta$) there must exist $\beta$ such that $Q[\alpha|\beta]Q[\beta|\alpha]<0$. In the other direction, as $\Delta_z[\beta|\beta] = \prod_{i \in \beta}z_i$, if $Q[\alpha|\beta]Q[\beta|\alpha]<0$ for some $\alpha,\beta$, then provided we choose $\tilde{z} \in \mathbb{R}^n_{+}$, such that $\tilde{z}_i=1$ for $i \in \beta$ and $z_i$ is sufficiently small for $i \not \in \beta$, then $\sum_\beta Q[\alpha|\beta]\,\Delta_{\tilde{z}}[\beta|\beta]\,Q[\beta|\alpha]<0$, and consequently $J^2[\alpha|\alpha]$ fails to be nonnegative for this choice of $\tilde{z} \in \mathbb{R}^n_{+}$. 
\end{enumerate}
This completes the proof.
\end{proof}

\begin{remark}[Conditions N1--N4 in the case of $(n,n+1,n)$ CRNs]
\label{remN1N4}
From Lemma~\ref{lemspecnohopf}, if $-(Q+Q^\mathrm{t})$ is a $P_0$ matrix, then condition N1 holds, while condition N2 is equivalent to sign-symmetry of $Q$. Both of these conditions involve only arithmetic computations. For $(n,n+1,n)$ CRNs, $\mathrm{det}\,J^{[2]}$ is a homogeneous polynomial of degree $n$ in the $n$ positive indeterminates $z_i$; when $n=3$ confirming whether $\mathrm{det}\,J^{[2]}$ has constant sign (condition N3) is, for most networks, a rapid and trivial computation. Finally, for $(n,n+1,n)$ CRNs, $\mathrm{det}(J^2+I)$ is a polynomial of degree $2n$ in the $n$ positive indeterminates $z_i$: in the case $n=3$ confirming either that this polynomial is positive on $\mathbb{R}^n_{+}$ (condition N4), or that it can take the value $0$, is often rapid. However, in some cases, this computation requires more effort.
\end{remark}

Computational implementation of the conditions detailed above gives us the following theorem about bimolecular $(3,4,3)$ systems. The code for the computations is available at \cite{muradgithub}.
\begin{thm}
\label{thms3r4Hopfpotential}
Checking conditions N1--N4 as outlined in Remark~\ref{remN1N4} we find that:
\begin{enumerate}
\item[(A)] Out of the $10853$ nondegenerate bimolecular $(3,4,3)$ CRNs identified in Lemma~\ref{lems3r4nondegen}, $Q$ fails to be negative semidefinite in $1599$. These fall into $1051$ dynamically non-equivalent classes. 

\item[(B)] Out of the $1599$ $(3,4,3)$ CRNs identified in (A), in $779$ $Q$ fails to be sign symmetric, and hence $J^2$ fails to be a $P_0$ matrix. These fall into $513$ dynamically non-equivalent classes.

\item[(C)] Out of the $779$ $(3,4,3)$ CRNs identified in (B), in $231$ we find that $\mathrm{det}\,J^{[2]}$ fails to have constant sign. These fall into $149$ dynamically non-equivalent classes.

\item[(D)] Out of the $231$ $(3,4,3)$ CRNs identified in (C), in $138$ we are unable to confirm that $\mathrm{det}(J^2+I) >0$ on the positive orthant. Moreover, for these $138$ we can easily confirm that $\mathrm{det}(J^2+I)=0$ occurs somewhere on the positive orthant. These fall into $87$ dynamically non-equivalent classes.
\end{enumerate}

\end{thm}

We illustrate the computations behind Theorem~\ref{thms3r4Hopfpotential} via a sequence of examples.

\begin{example}[Ruling out Hopf bifurcation in bimolecular $(3,4,3)$ networks]

We present four examples of nondegenerate, bimolecular, $(3,4,3)$, mass action CRNs, for which Hopf bifurcation is ruled out at stages (A), (B), (C) and (D) respectively in Theorem~\ref{thms3r4Hopfpotential}.

\begin{enumerate}[align=left,leftmargin=*]
\item[(A)] {\bf A network where $Q$ is negative semidefinite.}
\[
\mathsf{X} \rightarrow \mathsf{X}+\mathsf{Y}, \quad \mathsf{X}+\mathsf{Y} \rightarrow \mathsf{Z}, \quad \mathsf{Y}+\mathsf{Z} \rightarrow 2\mathsf{Z}, \quad 2\mathsf{Z} \rightarrow \mathsf{X}\,.
\]
We have
\[
\Gamma=\left(\begin{array}{rrrr}0&-1&0&1\\1&-1&-1&0\\0&1&1&-2\end{array}\right), \quad \Gamma_l=\left(\begin{array}{cccc}1&1&0&0\\0&1&1&0\\0&0&1&2\end{array}\right), \quad Q = \Gamma\,\Delta_u\,\Gamma_l^\mathrm{t} = \left(\begin{array}{rrr}-1&-1&2\\1&-2&-1\\1&2&-3\end{array}\right)
\]
(using $u:=(2,1,1,1)^\mathrm{t} \in \mathrm{ker}\,\Gamma$). We can now easily check that $-(Q+Q^\mathrm{t})$ is a $P_0$ matrix, and hence that $Q$ is negative semidefinite. Hopf bifurcation is ruled out.

\item[(B)] {\bf A network where $Q$ is sign-symmetric.}
\[
0 \rightarrow \mathsf{X}, \quad \mathsf{X}+\mathsf{Y} \rightarrow 2\mathsf{Y}, \quad 2\mathsf{X} \rightarrow \mathsf{X}+\mathsf{Z}, \quad \mathsf{Y}+\mathsf{Z} \rightarrow \mathsf{X}\,.
\]
We have
\[
\Gamma=\left(\begin{array}{rrrr}1&-1&-1&1\\0&1&0&-1\\0&0&1&-1\end{array}\right), \quad \Gamma_l=\left(\begin{array}{cccc}0&1&2&0\\0&1&0&1\\0&0&0&1\end{array}\right), \quad Q = \Gamma\,\Delta_u\,\Gamma_l^\mathrm{t} = \left(\begin{array}{rrr}-3&0&1\\1&0&-1\\2&-1&-1\end{array}\right)
\]
(using $u:=(1,1,1,1)^\mathrm{t} \in \mathrm{ker}\,\Gamma$). In this case, $-(Q+Q^\mathrm{t})$ fails to be a $P_0$ matrix; however $Q$ is sign-symmetric and consequently, for any positive diagonal matrix $\Delta$, $(Q\Delta)^2$ is a $P_0$ matrix. Hopf bifurcation is ruled out.

\item[(C)] {\bf A network where $\mathrm{det}\,J^{[2]}$ is signed.}
\[
0 \rightarrow\mathsf{X}, \quad \mathsf{X}+\mathsf{Y} \rightarrow 2\mathsf{Y}, \quad \mathsf{Y} \rightarrow\mathsf{Z}, \quad \mathsf{Y}+\mathsf{Z} \rightarrow 0\,.
\]
We have
\[
\Gamma=\left(\begin{array}{rrrr}1&-1&0&0\\0&1&-1&-1\\0&0&1&-1\end{array}\right), \quad \Gamma_l=\left(\begin{array}{cccc}0&1&0&0\\0&1&1&1\\0&0&0&1\end{array}\right), \quad Q = \Gamma\,\Delta_u\,\Gamma_l^\mathrm{t} = \left(\begin{array}{rrr}-2&-2&0\\2&0&-1\\0&0&-1\end{array}\right)
\]
(using $u:=(2,2,1,1)^\mathrm{t} \in \mathrm{ker}\,\Gamma$). In this case, $-(Q+Q^\mathrm{t})$ fails to be a $P_0$ matrix and $Q$ fails to be sign-symmetric. However, we can compute
\[
J=Q\,\Delta_z = \left(\begin{array}{ccc}-2z_1&-2z_2&0\\2z_1&0&-z_3\\0&0&-z_3\end{array}\right) \quad \mbox{and hence} \quad J^{[2]} = \left(\begin{array}{ccc}-2z_1&-z_3&0\\0&-2z_1-z_3&-2z_2\\0&2z_1&-z_3\end{array}\right)\,.
\]
We calculate that $\mathrm{det}\,J^{[2]} = -2z_1(z_3^2+2z_1z_3+4z_1z_2)$ which is clearly negative on $\mathbb{R}^3_{+}$. Hopf bifurcation is ruled out.

\item[(D)] {\bf A network where $\mathrm{det}\,(J^2+I)$ is positive.}
\[
0 \rightarrow \mathsf{X}, \quad \mathsf{X}+\mathsf{Y} \rightarrow 2\mathsf{Z}, \quad \mathsf{Y}+\mathsf{Z} \rightarrow 2\mathsf{Y}, \quad 2\mathsf{Z} \rightarrow \mathsf{Z}\,.
\]
We have
\[
\Gamma=\left(\begin{array}{rrrr}1&-1&0&0\\0&-1&1&0\\0&2&-1&-1\end{array}\right), \quad \Gamma_l=\left(\begin{array}{cccc}0&1&0&0\\0&1&1&0\\0&0&1&2\end{array}\right), \quad Q = \Gamma\,\Delta_u\,\Gamma_l^\mathrm{t} = \left(\begin{array}{rrr}-1&-1&0\\-1&0&1\\2&1&-3\end{array}\right)
\]
(using $u:=(1,1,1,1)^\mathrm{t} \in \mathrm{ker}\,\Gamma$). In this case, $-(Q+Q^\mathrm{t})$ fails to be a $P_0$ matrix, and $Q$ fails to be sign-symmetric. We can compute
\[
J=Q\,\Delta_z = \left(\begin{array}{ccc}-z_1&-z_2&0\\-z_1&0&z_3\\2z_1&z_2&-3z_3\end{array}\right) \quad \mbox{and hence} \quad J^{[2]} = \left(\begin{array}{ccc}-z_1&z_3&0\\z_2&-z_1-3z_3&-z_2\\-2z_1&-z_1&-3z_3\end{array}\right)\,.
\]
But $\mathrm{det}\,J^{[2]} = 3z_2z_3^2-9z_1z_3^2+2z_1z_2z_3-3z_1^2z_3+z_1^2z_2$ can take all signs on $\mathbb{R}^3_{+}$. However, we find that 
\[
\mathrm{det}\,(J^2+I) = (3z_1z_3-z_1z_2-z_2z_3-1)^2 + (2z_1z_2z_3+z_1+3z_3)^2
\]
which is clearly positive on $\mathbb{R}^3_{+}$ (one of the squares contains only positive terms). Hopf bifurcation is ruled out.
\end{enumerate}

\end{example}

\begin{remark}[Networks where Hopf bifurcation is ruled out, but oscillation occurs]
Up to dynamical equivalence, there are $19$ bimolecular $(3,4,3)$ networks where nondegenerate Hopf bifurcation is ruled out but an equilibrium can still have a pair of purely imaginary eigenvalues, and degenerate periodic orbits can occur. $15$ of these networks are nondegenerate, while four are degenerate. These networks will be examined further in future work. 
\end{remark}

\section{Confirming Hopf bifurcation in $(3,4,3)$ networks}
\label{sec343hopf}

The analysis so far has allowed us to rule out Hopf bifurcation in the great majority of  bimolecular $(3,4,3)$ networks. We now turn our attention to the 87 dynamically non-equivalent CRNs with the potential for Hopf bifurcation identified in Theorem~\ref{thms3r4Hopfpotential}. We wish to know which of these networks admit nondegenerate Hopf bifurcation, unfolded by the rate constants; and amongst these, which networks admit a supercritical bifurcation, subcritical bifurcation, or both. 

We outline the theory only briefly, focussed on the case of interest here, and referring the reader to \cite{kuznetsov:2004} for the details. Let $U \subseteq \mathbb{R}^n$ be open and consider the parameterised system
\begin{align}
\label{eq:hopf_ode}
\dot{x} = f(x,\alpha), \hspace{0.5cm} x \in U,\,\, \alpha \in \mathbb{R}\,.
\end{align}
Assume that $f$ is sufficiently smooth, that $f(0,0)=0$, and that $D_xf(0,0)$ has a pair of nonzero imaginary eigenvalues. To simplify matters, we also assume that all other eigenvalues of $D_xf(0,0)$ have negative real parts, as this is the situation which must occur, for reasons discussed further below, in all $3$-species, $4$-reaction, bimolecular, mass action networks with potential Hopf bifurcation. Since the origin is a nondegenerate equilibrium when $\alpha=0$, we may also assume, without loss of generality (i.e., via application of the implicit function theorem), that $f(0,\alpha)=0$ for all $\alpha$ sufficiently close to $0$. 

Confirming that a nondegenerate Hopf bifurcation indeed occurs at $(0,0)$ and is unfolded by the parameter $\alpha$ requires us to show two things:
\begin{enumerate}[align=left,leftmargin=*]
\item {\bf Nondegeneracy.} The \emph{first Lyapunov coefficient}, denoted by $l_1(0)$ and to be defined in Section~\ref{secnondegen}, is nonzero. 
\item {\bf Transversality.} A pair of eigenvalues of $D_xf(0,\alpha)$ moves with nonzero speed through the imaginary axis as $\alpha$ crosses $0$. 
\end{enumerate}
If these conditions are met, then there exists $\alpha_0>0$ such that, after sending $\alpha \mapsto -\alpha$ if necessary, one of the following two situations occurs:
\begin{itemize}[align=left,leftmargin=*]
\item {\bf A supercritical Hopf bifurcation.} If $l_1(0) < 0$, then for $-\alpha_0 <\alpha \leq 0$ the equilibrium at the origin is asymptotically stable; for $0<\alpha<\alpha_0$ it is unstable, and there exists a small asymptotically stable periodic orbit close to the origin.
\item {\bf A subcritical Hopf bifurcation.} If $l_1(0) > 0$, then for $0 \leq \alpha < \alpha_0$ the equilibrium at the origin is unstable; for $-\alpha_0 < \alpha < 0$ it is asymptotically stable, while nearby is a small unstable periodic orbit.
\end{itemize}

\begin{remark}[Transversality in the case of several parameters]
\label{remtrans}
In the case that there are several parameters, let us say $\alpha_1, \ldots, \alpha_k$, in \eqref{eq:hopf_ode} we say that the bifurcation is unfolded by the parameters if it is unfolded by some $\alpha_i$. We can also often phrase the condition for transversality in a practically useful way in terms of the regularity of a map. Suppose $x \in U \subseteq \mathbb{R}^n, \,\,\alpha \in V \subseteq \mathbb{R}^k$ ($U$ and $V$ are assumed to be open) and we have a sufficiently smooth parameterised ODE system $\dot x = f(x, \alpha)$ on $U$. Define $\mathrm{H}$ to be the set of points in $U \times V$ satisfying the basic conditions for Hopf bifurcation: (i) $f(x,\alpha)=0$, and (ii) $D_xf(x,\alpha)$ has a pair of imaginary eigenvalues and no other eigenvalues on the imaginary axis. Suppose also that we can write down a (sufficiently smooth) function $g: U \times V \to \mathbb{R}$ such that any $(\tilde x, \tilde \alpha) \in \mathrm{H}$ has some neighbourhood in which $\mathrm{H}$ coincides with $f^{-1}(0) \cap g^{-1}(0)$; and such that $g$ changes sign on $f^{-1}(0)$ if and only if a pair of eigenvalues of $D_xf(x,\alpha)$ cross the imaginary axis. Then the parameters $\alpha$ unfold the bifurcation at $(\tilde x, \tilde \alpha)$ if and only if the map $(x, \alpha) \mapsto (f(x, \alpha), g(x,\alpha))$ is regular at $(\tilde x, \tilde \alpha)$. A natural choice for the function $g$ is given in Lemma~\ref{lemtransverse} below.
\end{remark}

We now turn back to the 87 non-equivalent bimolecular $(3,4,3)$ CRNs with the potential for Hopf bifurcation. Consider the ODE system $\dot x = f(x, \kappa) := \Gamma (\kappa \circ x^{\Gamma_l^\mathrm{t}})$ associated with one of these networks. The {\bf Hopf set} for the network will mean the subset of $\mathbb{R}^3_{+} \times \mathbb{R}^4_{+}$ satisfying (i) $f(x,\kappa)=0$, and (ii) $D_xf(x,\kappa)$ has a pair of nonzero imaginary eigenvalues. In Theorem~\ref{thms3r4Hopfpotential} we have already confirmed that all 87 networks have nonempty Hopf set. Further, all the networks are nondegenerate, and so, by Lemma~\ref{lemnondegen}, all equilibria associated with the Hopf set of each network are nondegenerate. Moreover, in every case, $D_xf(x,\kappa)$ has a real, negative eigenvalue on the Hopf set. That this must occur follows immediately from the facts that the divergence of a bimolecular mass action system is nonpositive at any positive equilibrium (see the proof of Theorem~6 in \cite{boros:hofbauer:2022a}), and that $D_xf(x,\kappa)$ is nonsingular. In fact, as the eigenvalues of $D_xf(x,\kappa)$ depend continuously on $(x,\kappa)$, and $D_xf(x,\kappa)$ is nonsingular at positive equilbria, it must have a real negative eigenvalue at each positive equilibrium, whether on the Hopf set or not.

In order to decide whether a nondegenerate Hopf bifurcation can occur in each of the 87 networks, we compute, using symbolic algebra as described in Section~\ref{secnondegen}, the first Lyapunov coefficient at each point on the Hopf set for each network. (Recall that for a nondegenerate Hopf bifurcation to occur we require the first Lyapunov coefficient to be nonzero.) We also check, as described in more detail in Section~\ref{sectrans} below, that the rate constants unfold the bifurcation. 

To following theorem summarises the outcome of this process. The full code and detailed output of the computations are available on GitHub \cite{balazsgithub,muradgithub}.
\begin{thm}
\label{thms3r4Hopf}
Nondegenerate Hopf bifurcations occur in $86$ of the $87$ dynamically non-equivalent, bimolecular $3$-species, $4$-reaction CRNs with the potential for Hopf bifurcation identified in Theorem~\ref{thms3r4Hopfpotential}. A nondegenerate supercritical Hopf bifurcation occurs in $57$ of these $86$ networks; a nondegenerate subcritical Hopf bifurcation occurs in $54$ of these $86$ networks; and both supercritical and subcritical bifurcations occur in $25$ of the networks. The single remaining network does not admit a nondegenerate Hopf bifurcation: the first Lyapunov coefficient is identically zero on the Hopf set. All the Hopf bifurcations, both nondegenerate and degenerate, are unfolded by the rate constants. 
\end{thm}

The full list of networks admitting nondegenerate Hopf bifurcation in Theorem~\ref{thms3r4Hopf} appears in Appendix~\ref{appnetworks}. It is helpful to classify the $86$ non-equivalent CRNs capable of nondegenerate Hopf bifurcation into groups. As expected, all of these networks have $4$ distinct reactant complexes (see Remark~\ref{remND}), and we can group the networks according to the molecularities of these reactant complexes. A network whose reactant complexes have molecularities $a$, $b$, $c$ and $d$ will be said to have {\bf reactant molecularity} $(a,b,c,d)$, where we assume that $a \leq b \leq c \leq d$. 

There are four reactant molecularities which occur amongst the 86 networks admitting nondegenerate Hopf bifurcation: $(0,1,2,2)$, i.e., one constant, one linear and two quadratic reaction rates; $(1,1,2,2)$, i.e., two linear and two quadratic reaction rates; $(0,2,2,2)$, i.e., one constant and three quadratic reaction rates; and $(1,2,2,2)$, i.e., one linear and three quadratic reaction rates. In Table~\ref{table86summary}, we present a classification of these 86 networks according to reactant molecularity and the signs that the first Lyapunov coefficient, denoted by $L_1$, can take on the Hopf set. 

We also observe that the exceptional network admitting only degenerate Hopf bifurcation has reactant molecularity $(0,2,2,2)$. There can thus be no Hopf bifurcation, degenerate or nondegenerate, in a bimolecular $(3,4,3)$ mass action system with only a single quadratic term in its ODEs. This is consistent with the results of Wilhelm and Heinrich \cite{WilhelmHeinrich1995,WilhelmHeinrich1996}, who found that the smallest bimolecular, $3$-species, mass action system admitting Hopf bifurcation, and having only one quadratic term in its ODEs, must have five reactions.

\bgroup
\def\arraystretch{1.5}
\begin{table}[h]
\begin{center}
\begin{tabular}{c"c|c|c|c|}
\multicolumn{1}{c}{}&\multicolumn{4}{c}{Reactant molecularity}\\
        & $(0,1,2,2)$ & $(1,1,2,2)$ & $(0,2,2,2)$ & $(1,2,2,2)$ \\
\thickhline
$L_1<0$ & $5$ & $8$ & $5$ & $14$ \\\hline
$L_1 \gtreqless 0$ & $1$ & $6$ & $0$ & $18$ \\\hline
$L_1>0$ & $1$ & $7$ & $0$ & $15$ \\\hline
$L_1 \geq 0$ & $6$ & $0$ & $0$ & $0$ \\
\hline
\end{tabular}
\end{center}
\caption{\label{table86summary}Total numbers of non-equivalent bimolecular $(3,4,3)$ networks admitting nondegenerate Hopf bifurcation, classified according to reactant molecularity and the signs that the first Lyapunov coefficient can take at bifurcation. We see, for example, that there are $6$ networks with reactant molecularity $(1,1,2,2)$ (i.e., two linear and two quadratic reaction rates) which additionally have the property that $L_1$ can take all signs (i.e., both supercritical and subcritical bifurcations can occur). The networks themselves falling into each category are listed in Appendix~\ref{appnetworks}.}
\end{table}
\egroup

\subsection{Nondegeneracy of Hopf bifurcations}
\label{secnondegen}
We outline the computation of the first Lyapunov coefficient $l_1(0)$ for \eqref{eq:hopf_ode} following \cite{kuznetsov:2004}. Denote by $A(\alpha) \in \mathbb{R}^{n\times n}$ the Jacobian matrix $D_xf(0,\alpha)$. We assume that eigenvalues of $A(0)$ include the purely imaginary pair $\pm i \omega$ where $\omega > 0$. We can write the Taylor expansion of $f(x,0)$ as
\begin{align*}
f(x,0) = A(0) x + \frac12 B(x,x) + \frac16 C(x,x,x) + O(\|x\|^4),
\end{align*}
where $B\colon \mathbb{R}^n \times \mathbb{R}^n \to \mathbb{R}^n$ and $C\colon \mathbb{R}^n \times \mathbb{R}^n \times \mathbb{R}^n \to \mathbb{R}^n$ are the multilinear functions
\begin{align*}
B_j(x,y) = \sum_{k,l=1}^n \frac{\partial^2 f_j(\xi,0)}{\partial \xi_k \partial \xi_l}\bigg|_{\xi=0} x_k y_l, \quad
C_j(x,y,z) = \sum_{k,l,m=1}^n \frac{\partial^3 f_j(\xi,0)}{\partial \xi_k \partial \xi_l \partial \xi_m}\bigg|_{\xi=0}x_k y_l z_m
\end{align*}
for $j=1,\ldots,n$. Further, let $p,q\in\mathbb{C}^n$ be left and right eigenvectors of $A(0)$ satisfying
\begin{align*}
A(0) q = \omega i q, \quad
A(0)^\mathrm{t} p = -\omega i p, \quad
\langle p, q \rangle = 1.
\end{align*}
Here, $\langle \cdot, \cdot \rangle \colon \mathbb{C}^n \times \mathbb{C}^n \to \mathbb{C}$ is the standard scalar product in $\mathbb{C}^n$. With all this preparation, the first Lyapunov coefficient, $l_1(0)$, is given by
\begin{align}
\label{eql1}
l_1(0)=\frac{1}{2\omega}\mathrm{Re} \left\langle p,C(q,q,\overline{q})+2B\left(q,(-A(0))^{-1}B(q,\overline{q})\right) + B\left(\overline{q},(2\omega i I\,-A(0))^{-1}B(q,q)\right)\right\rangle,
\end{align}
where $I\,$ is the $n$ by $n$ identity matrix.

A concrete example of the calculation of $l_1(0)$ is presented in Section~\ref{secworked}.

\subsection{Transversality of Hopf bifurcations}
\label{sectrans}

All of the nondegenerate Hopf bifurcations described above in bimolecular $(3,4,3)$ networks, and the single degenerate Hopf bifurcation too, are easily checked to be unfolded by the rate constants. Confirming this is facilitated by the following lemma. 

\begin{lemma}
\label{lemtransverse}
Consider a nondegenerate, dynamically nontrivial, $(n,n+1,n)$ CRN giving rise to the differential equation $\dot x = f(x, \kappa) := \Gamma (\kappa \circ x^{\Gamma_l^\mathrm{t}})$. Suppose that $(\tilde{x}, \tilde{\kappa})$ satisfies:
\begin{itemize}
\item $\tilde{x}$ is a nondegenerate equilibrium of $\dot x = f(x, \tilde{\kappa})$, i.e., $f(\tilde{x}, \tilde{\kappa}) = 0$ and $\mathrm{det}\,D_xf(\tilde{x}, \tilde{\kappa}) \neq 0$.
\item $D_xf(\tilde{x}, \tilde{\kappa})$ has a pair of nonzero imaginary eigenvalues, and all other eigenvalues of $D_xf(\tilde{x}, \tilde{\kappa})$ have negative real parts.
\end{itemize}
Fix some positive $u \in \mathrm{ker}\,\Gamma$ and define $g(x) := \mathrm{det}\,(\Gamma \Delta_{u} \Gamma_l^{\mathrm{t}}\Delta_{1/x})^{[2]}$. Then the bifurcation at $(\tilde{x}, \tilde{\kappa})$ is unfolded by the rate constants $\kappa$ if and only if $g$ is regular at $\tilde{x}$, i.e., $Dg(\tilde{x}) \neq 0$.
\end{lemma}
\begin{proof}
Let $(\tilde{x}, \tilde{\kappa})$ satisfy the conditions of the lemma. Observe that the Hopf set is defined, in some neighbourhood of $(\tilde{x}, \tilde{\kappa})$, by the conditions $f(x,\kappa)=0$ and $G(x, \kappa):=\mathrm{det}\,((D_xf(x, \kappa))^{[2]}) = 0$, where $D_xf(x, \kappa)=\mu(\kappa) \Gamma \Delta_{u} \Gamma_l^{\mathrm{t}}\Delta_{1/x}$. Here $\mu(\kappa)$ is the positive scalar function of the rate constants obtained by solving (\ref{eqxmukappa}). Moreover $G$ changes sign on $f^{-1}(0)$ if and only if a pair of eigenvalues cross the imaginary axis. The same observations clearly hold for $g(x) := G(x, \kappa)/\mu(\kappa)$. Thus transversality of the Hopf bifurcation at $(\tilde{x}, \tilde{\kappa})$ is equivalent to regularity of $h(x,\kappa):=(f(x,\kappa),g(x))$ at $(\tilde{x}, \tilde{\kappa})$ (see Remark~\ref{remtrans}).

The Jacobian matrix of $h$ has dimensions $(n+1) \times (2n+1)$ and takes the form
\[
Dh:=\left(\begin{array}{cc}D_xf&D_\kappa f\\Dg&0\\\end{array}\right)\,, \quad \mbox{so that} \quad Dh(\tilde{x},\tilde{\kappa}) = \left(\begin{array}{cc}\mu(\tilde{\kappa}) \Gamma \Delta_{u} \Gamma_l^\mathrm{t} \Delta_{1/\tilde{x}}&\mu(\tilde{\kappa})\Gamma \Delta_{u/\tilde{\kappa}}\\Dg(\tilde{x})&0\end{array}\right)\,.
\]
Since the network is nondegenerate by assumption, $D_xf(x, \kappa) = \mu(\tilde{\kappa}) \Gamma \Delta_{u} \Gamma_l^\mathrm{t} \Delta_{1/\tilde{x}}$ has rank $n$. On the other hand, the $n \times (n+1)$ matrix $\mu(\tilde{\kappa})\Gamma \Delta_{u/\tilde{\kappa}}$ clearly has rank $n$, being the product of $\Gamma$, which has rank $n$, and a positive diagonal matrix. We can conclude that $Dh(\tilde{x},\tilde{\kappa})$ has rank $n+1$ if and only if $Dg(\tilde{x})\neq 0$. In one direction this is trivial since if $Dg(\tilde{x}) = 0$, then $\mathrm{rank}\,Dh(\tilde{x},\tilde{\kappa}) < n+1$. In the other direction, if $Dg(\tilde{x}) \neq 0$, the bottom row of $Dh(\tilde{x},\tilde{\kappa})$ clearly cannot be a linear combination of the remaining rows, as $\mu(\tilde{\kappa})\Gamma \Delta_{u/\tilde{\kappa}}$ has rank $n$; hence $\mathrm{rank}\,Dh(\tilde{x},\tilde{\kappa}) = n+1$.
\end{proof}

\begin{remark}
Note that the transversality condition in Lemma~\ref{lemtransverse} does not depend directly on the rate constants $\kappa$ at all. Moreover, we do not require the bifurcation to be nondegenerate, beyond the requirement that the equilibrium itself is nondegenerate. 
\end{remark}

In the light of Lemma~\ref{lemtransverse}, for any nondegenerate $(n,n+1,n)$ CRN, we define $g(z) := \mathrm{det}\,(Q\Delta_z)^{[2]}$, where, as usual, we have chosen some positive $u \in \mathrm{ker}\,\Gamma$, and set $Q=\Gamma \Delta_u\Gamma_l^\mathrm{t}$. We can now confirm, for each of the $87$ bimolecular $(3,4,3)$ networks admitting Hopf bifurcation, including the exceptional network admitting only a degenerate bifurcation, that $g(z)$ is regular on $\mathbb{R}^3_{+}$. The code used to confirm this is on GitHub \cite{muradgithub}. Thus in all cases the bifurcations are unfolded by the rate constants. 

An example illustrating the nondegeneracy and transversality calculations is presented next.

\section{A worked example}
\label{secworked}

We consider the bimolecular $(3,4,3)$ network
\[
0 \overset{\kappa_1}{\longrightarrow} \mathsf{X}, \quad \mathsf{X} \overset{\kappa_2}{\longrightarrow} \mathsf{Y}, \quad \mathsf{Y}+\mathsf{Z} \overset{\kappa_3}{\longrightarrow} 2\mathsf{Z}, \quad \mathsf{X}+\mathsf{Z} \overset{\kappa_4}{\longrightarrow} 0\,,
\]
with mass action rate constants $\kappa_1, \kappa_2, \kappa_3$ and $\kappa_4$ as shown. This is network 1a listed in Appendix~\ref{appnetworks}. 

\subsection{The original and the recoordinatised ODE} 
The network gives rise to the system of ODEs
\begin{align}\label{eq:ode_xyz}
\begin{split}
\dot{x} &= \kappa_1-\kappa_2x-\kappa_4xz, \\
\dot{y} &= \kappa_2 x-\kappa_3yz,\\
\dot{z} &= \kappa_3yz - \kappa_4xz\,,
\end{split}
\end{align}
where $x,y$ and $z$ are the concentrations of $\mathsf{X}$, $\mathsf{Y}$ and $\mathsf{Z}$ respectively. The recoordinatisation described in Section~\ref{secrecoord} transforms \eqref{eq:ode_xyz} into
\begin{align}\label{eq:ode_uvw}
\begin{split}
\dot{u} &= \alpha(2 - u - u w), \\
\dot{v} &= \beta(u - v w),\\
\dot{w} &= \gamma(v w- u w),
\end{split}
\end{align}
where $\alpha = \kappa_2$, $\beta = \frac{\kappa_2\kappa_3}{\kappa_4}$, $\gamma = \frac{\kappa_1\kappa_4}{2\kappa_2}$, and $u, v$ and $w$ are the rescaled concentrations of $\mathsf{X}$, $\mathsf{Y}$ and $\mathsf{Z}$ respectively. The Jacobian matrix of \eqref{eq:ode_uvw} evaluated at the unique positive equilibrium $(1,1,1)$ equals
\[
A = \left(\begin{array}{rrr}
-2\alpha & 0 & -\alpha \\
\beta & -\beta & -\beta \\
-\gamma & \gamma & 0\end{array}\right).
\]
Since $\det A = -2\alpha\beta\gamma<0$, the unique positive equilibrium is nondegenerate, and the matrix $A$ always has a negative real eigenvalue. 

\subsection{The Hopf set} 
The characteristic polynomial of $A$ is
\begin{align*}
\lambda^3 + (2\alpha +\beta)\lambda^2 + (2\alpha \beta - \alpha \gamma + \beta \gamma)\lambda + 2\alpha\beta\gamma.
\end{align*}
We can calculate that $A$ has a pair of nonzero imaginary eigenvalues if and only if
\[
(\alpha, \beta, \gamma) \in \mathrm{H}:=\left\{(\alpha,\beta,\gamma) \in \mathbb{R}^3_{+}\,\colon\,\beta<2\alpha \text{ and } \gamma=\frac{2\alpha\beta(2\alpha+\beta)}{(\alpha+\beta) (2\alpha-\beta)}\right\}\,.
\]
Note that the Hopf set of \eqref{eq:ode_uvw} is, strictly speaking, $\{(1,1,1)^\mathrm{t}\} \times \mathrm{H}$; but it should cause no confusion to refer to $\mathrm{H}$ as the Hopf set.

To simplify some calculations, we can, without loss of generality, set $\alpha = 1$; this is equivalent to rescaling time $t \mapsto \kappa_2 t$, and redefining $\beta = \frac{\kappa_3}{\kappa_4}$ and $\gamma = \frac{\kappa_1\kappa_4}{2\kappa_2^2}$. The section of $\mathrm{H}$ satisfying $\alpha=1$ is
\begin{align}\label{eq:RH_beta_gamma}
\widehat{\mathrm{H}}:=\left\{(\beta,\gamma) \in \mathbb{R}^2_{+}\,\colon\,\beta<2 \text{ and } \gamma = \frac{2\beta(\beta+2)}{(\beta+1) (2-\beta)}\right\}.     
\end{align}

\subsection{Nondegeneracy of the bifurcation} 
We carry out the computations detailed in Section~\ref{secnondegen}. The goal is to calculate the first Lyapunov coefficient, denoted by $L_1$, at an arbitrary point on the Hopf set. After shifting the equilibrium $(1,1,1)$ of \eqref{eq:ode_uvw} to the origin, the differential equation becomes
\begin{align} \label{eq:ode_uvw_origin}
\left(\begin{array}{c}
\dot u \\ \dot{v} \\ \dot{w}
\end{array}\right) = 
A \left(\begin{array}{c}
u \\ v \\ w
\end{array}\right) + 
\left(\begin{array}{c}
-\alpha uw \\ -\beta vw \\ \gamma vw - \gamma uw
\end{array}\right)
\end{align}
where we have kept the notation $u$, $v$, $w$ for the shifted variables. Now let $B \colon \mathbb{R}^3 \times \mathbb{R}^3 \to \mathbb{R}^3$ be the symmetric bilinear function for which the quadratic terms in the r.h.s.\ of the differential equation \eqref{eq:ode_uvw_origin} equals 
\begin{align*}
\frac12 B\left(\begin{bmatrix} u \\ v \\ w \end{bmatrix},\begin{bmatrix} u \\ v \\ w \end{bmatrix}\right),
\end{align*}
i.e.,
\begin{align*}
B \colon \left(\begin{bmatrix} x_1 \\ x_2 \\ x_3 \end{bmatrix},\begin{bmatrix} y_1 \\ y_2 \\ y_3 \end{bmatrix}\right) \mapsto \left(\begin{array}{c}
-\alpha(x_1 y_3 + x_3 y_1)\\
-\beta(x_2 y_3 + x_3 y_2) \\
\gamma(x_2 y_3 + x_3 y_2) - \gamma(x_1 y_3 + x_3 y_1)
\end{array}\right).
\end{align*}
Notice that due to the bimolecularity of the mass action system \eqref{eq:ode_xyz}, there are no cubic or higher order terms on the r.h.s.\ of \eqref{eq:ode_uvw_origin}, and hence the multilinear function $C$ which figures in \eqref{eql1} is zero. When the eigenvalues of $A$ are of the form $\{\omega i, -\omega i, \varrho\}$ for some $\omega>0$ and $\varrho<0$, we have $\omega = \sqrt{\frac{\det A}{\mathrm{tr}\, A}}$. Let $p, q \in \mathbb{C}^3$ be left and right eigenvectors of $A$ satisfying
\begin{align*}
A q = \omega i q, \quad A^\mathrm{t} p = -\omega i p, \quad \langle p, q \rangle = 1\,.
\end{align*}
Using the general formula in \eqref{eql1}, the first Lyapunov coefficient is now, up to positive scaling,
\begin{align*}
L_1 = \mathrm{Re}\, \left\langle p,2B\left(q,(-A)^{-1}B(q,\overline{q})\right) + B\left(\overline{q},(2\omega i I-A)^{-1}B(q,q)\right)\right\rangle.
\end{align*}
(The dependence on $\alpha, \beta$ and $\gamma$ has been suppressed for notational convenience.) 

In order to determine the sign of $L_1$ on the Hopf set, we only need to determine its sign on $\widehat{\mathrm{H}}$. Computing $\omega$ and eigenvectors $p$ and $q$ in terms of $\beta$ and $\gamma$, we find that on $\widehat{\mathrm{H}}$, $L_1$ equals, up to scaling,
\begin{align*}
-\frac{9216 \beta^{11}(\beta+2)^4(3-\beta)^2 (\beta^6+15\beta^5-106\beta^4-8\beta^3+320\beta^2+176\beta+96)}{(\beta-2)^9(\beta+1)^8}.
\end{align*}
We find that $L_1|_{\widehat{\mathrm{H}}}$ is negative for all $0<\beta<2$, and hence the same holds for $L_1$ on all of $\mathrm{H}$. It follows that the equilibrium $(1,1,1)$ of \eqref{eq:ode_uvw} is asymptotically stable for all the parameters on $\mathrm{H}$. Furthermore, for $\gamma$ slightly larger than $\frac{2\alpha\beta(2\alpha+\beta)}{(\alpha+\beta) (2\alpha-\beta)}$, a stable limit cycle, born via a supercritical Hopf bifurcation, exists on the unstable manifold of $(1,1,1)$.

\begin{remark}[A minimal parameterisation of nondegenerate $(n,n+1,n)$ networks]
\label{remHcurve}
For all nondegenerate $(n,n+1,n)$ networks, the coordinate change carried out above (see Section~\ref{secrecoord}), followed by the rescaling of time which led to the fixing of one parameter, are both possible. We may thus consider nondegenerate, $(3,4,3)$, mass action networks to be parameterised by exactly two parameters, and for any such network, we can reasonably consider the Hopf set projected onto parameter space as a ``Hopf curve'', such as $\widehat{\mathrm{H}}$ in \eqref{eq:RH_beta_gamma} above, in a two dimensional parameter space consisting of two of the three parameters $\alpha, \beta$ and $\gamma$.
\end{remark}

\subsection{Unfolding by the rate constants} 
Next, we use this example to illustrate the application of Lemma~\ref{lemtransverse} to prove that the Hopf bifurcation is unfolded by the rate constants. In this case the stoichiometric matrix $\Gamma$, and left stoichiometric matrix $\Gamma_l$ are
\[
\Gamma = \left(\begin{array}{rrrr}1&-1&0&-1\\0&1&-1&0\\0&0&1&-1\end{array}\right) \quad \mbox{and} \quad \Gamma_l = \left(\begin{array}{cccc}0&1&0&1\\0&0&1&0\\0&0&1&1\end{array}\right)\,.
\]
We can choose $u = (2,1,1,1)^\mathrm{t} \in \mathrm{ker}\,\Gamma$, and set $z = 1/x$ to obtain
\[
\Gamma \Delta_u \Gamma_l^\mathrm{t}\Delta_z = \left(\begin{array}{ccc}-2z_1&0&-z_3\\z_1&-z_2&-z_3\\-z_1&z_2&0\end{array}\right), \quad (\Gamma \Delta_u \Gamma_l^\mathrm{t}\Delta_z)^{[2]} = \left(\begin{array}{ccc}-2z_1-z_2&-z_3&z_3\\z_2&-2z_1&0\\z_1&z_1&-z_2\end{array}\right)
\]
and hence $g(z) = \mathrm{det}\,((\Gamma \Delta_u \Gamma_l^\mathrm{t}\Delta_z)^{[2]}) = -4z_1^2z_2+z_3z_1z_2-z_3z_2^2-2z_1z_2^2+2z_3z_1^2$. What remains to be shown is that $g$, regarded as a map from $\mathbb{R}^3_{+}$ to $\mathbb{R}$, is everywhere regular. This is easily shown, for example by proving that $\left(\frac{\partial g}{\partial z_1}\right)^2 + \left(\frac{\partial g}{\partial z_2}\right)^2 + \left(\frac{\partial g}{\partial z_3}\right)^2$ is (strictly) positive on $\mathbb{R}^3_{+}$.

We have thus shown that this network undergoes a supercritical bifurcation at each point on its Hopf set, and that this bifurcation is unfolded by the rate constants. 

We obtain the results of Theorem~\ref{thms3r4Hopf} precisely by carrying out such computations on each of the 87 non-equivalent networks given by Theorem~\ref{thms3r4Hopfpotential}. The networks themselves, and the outcomes of the calculations of the first Lyapunov coefficient, are in Appendix~\ref{appnetworks}. The code used to carry out the computations is on GitHub \cite{balazsgithub, muradgithub}.

\section{Bifurcations of higher codimension and the creation of multiple periodic orbits}
\label{secBautin}

In this section, we briefly discuss bifurcations of higher codimension in the 86 bimolecular $(3,4,3)$ networks which admit nondegenerate Hopf bifurcation with mass action kinetics. We first note that none of the 86 networks can admit a Bogdanov--Takens bifurcation, as equilibria of these networks are all nondegenerate. Indeed, in light of the dichotomy in Lemma~\ref{lemnondegen}, $(n,n+1,n)$ CRNs cannot admit even fold bifurcations. 

So when considering bifurcations of codimension $2$, the possibility of interest is a generalised Hopf bifurcation, also known as a Bautin bifurcation \cite[Section 8.3]{kuznetsov:2004}. In particular we would like to know: can any of the networks have multiple periodic orbits for some values of the rate constants? And can a stable equilibrium coexist with a stable periodic orbit?

The main bifurcation condition for a Bautin bifurcation is that the first Lyapunov coefficient, which we denoted by $L_1$, must vanish at a point on the Hopf set. But this condition alone is not sufficient to guarantee that a nondegenerate Bautin bifurcation occurs and is unfolded by the rate constants: we must additionally confirm nondegeneracy and transversality conditions. 

We omit most of the detail, but observe that the nondegeneracy condition corresponds to the nonvanishing of the {\em second} Lyapunov coefficient, denoted $L_2$, whose derivation for systems of dimension greater than $2$ is detailed in \cite[Sections 8.7.1 and 8.7.3]{kuznetsov:2004}. The transversality condition can, as usual, be phrased in terms of regularity of a certain map. 

We were able to confirm the basic Bautin bifurcation condition, namely the vanishing of $L_1$ on the Hopf set, in 31 of the 86 networks admitting Hopf bifurcation. Let us say that these $31$ networks display ``potential Bautin bifurcation'', and refer to the subset of the Hopf set where $L_1=0$ as the {\bf Bautin set}. Out of the 31 networks with potential Bautin bifurcation, $L_1$ can take all signs in 25 (the networks in the second row in Tables~\ref{table86summary}~and~\ref{table86classify}); and $L_1$ is nonnegative, but can definitely be zero, in the remaining $6$ (the final row in Tables~\ref{table86summary}~and~\ref{table86classify}). 

For all 31 networks with potential Bautin bifurcation, we were able to confirm that $L_2$ is always nonzero on the Bautin set, so the main nondegeneracy condition for Bautin bifurcation is satisfied. In 29 of the networks (including one network where the Bautin set has two distinct components -- see Note~\ref{noteL1zerotwice} in Appendix~\ref{appnetworks}), $L_2<0$ on the Bautin set. In the remaining two networks $L_2>0$ on the Bautin set. The Mathematica code for all of these computations is available on GitHub \cite{balazsgithub}.

For the six networks where $L_1 \geq 0$, it is immediate that the potential Bautin bifurcation cannot be unfolded by the rate constants (no supercritical Hopf bifurcations can occur nearby). In the remaining $25$ networks, we were able to confirm the transversality condition, i.e., the Bautin bifurcation is, indeed, unfolded by the rate constants. Altogether, we are able to make the following claims.

\begin{thm}
\label{thmBautin}
Up to dynamical equivalence, there are $31$ bimolecular $3$-species, $4$-reaction mass action networks with potential Bautin bifurcation. Of these, $25$ admit a nondegenerate Bautin bifurcation: amongst the behaviours which must occur in these networks is the creation of a stable and unstable periodic orbit in a fold bifurcation. The remaining $6$ networks, where $L_1$ is nonnegative, also admit multiple periodic orbits for some choices of rate constants. Further,
\begin{enumerate}[align=left,leftmargin=*]
\item For $29$ networks where $L_2 < 0$ on the Bautin set, there exist rate constants such that a linearly stable equilibrium coexists with a linearly stable periodic orbit. In particular, for some rate constants, there exists a locally invariant two dimensional manifold containing an asymptotically stable equilibrium surrounded by two periodic orbits: the inner one is repelling, and the outer one is attracting. 
\item For $2$ networks where $L_2 > 0$ on the Bautin set, there exist rate constants where an unstable equilibrium coexists with a linearly stable periodic orbit and an unstable periodic orbit. In particular, for some rate constants, there exists a locally invariant two dimensional manifold containing an unstable equilibrium surrounded by two periodic orbits: the inner one is attracting, and the outer one is repelling. 
\end{enumerate}
The networks which display these behaviours are listed in Note~\ref{noteBautin} in Appendix~\ref{appnetworks}.
\end{thm}

For the 25 networks which admit a nondegenerate Bautin bifurcation, the claims in Theorem~\ref{thmBautin} follow immediately from standard theory (see \cite[Section 8.3]{kuznetsov:2004}). For the remaining six networks, the claims again largely follow constructions in \cite{kuznetsov:2004}. We omit a full proof, but sketch the arguments. Note first that in all six of these cases, $L_2<0$ on the Bautin set. Fix a network with potential Bautin bifurcation, let $\mathrm{H}$ be the Hopf set, and let $(x_0, \kappa_0) \in \mathrm{B} \subseteq \mathrm{H}$, where $\mathrm{B}$ is the Bautin set (on which we assume $L_2<0$). The key step is then to show, from careful examination of the Poincar\'e map associated with the normal form of the system on the parameter dependent center manifold of $(x_0, \kappa_0)$, that for any $(x_1, \kappa_1) \in \mathrm{H}$ sufficiently close to $(x_0, \kappa_0)$ and satisfying $L_1>0$, the center manifold of $x_1$ includes a linearly stable periodic orbit, say $\mathcal{O}_s$, surrounding $x_1$. On the other hand, a standard subcritical Hopf bifurcation occurs at $(x_1, \kappa_1)$, i.e., $(x_1, \kappa_1)$ has a neighbourhood, say $V$, which includes points $(x_2, \kappa_2)$ where now $x_2$ is linearly stable and is surrounded by an unstable periodic orbit, say $\mathcal{O}_u$. Provided $V$ is small enough, $\mathcal{O}_s$ continues to exist at rate constants $\kappa_2$. Thus we obtain, on the parameter dependent center manifold, a stable equilibrium $x_2$, an unstable periodic orbit $\mathcal{O}_u$, and a stable periodic orbit $\mathcal{O}_s$, for some choices of parameters close to $\kappa_0$, provided $(x_0, \kappa_0)$ has a neighbourhood in $\mathrm{H}$ which includes points where $L_1>0$. 

\begin{example}[Coexistence of a stable equilibrium and a stable periodic orbit]
We consider the following mass action network 
\[
\mathsf{X} \overset{\kappa_1}{\longrightarrow} 2\mathsf{X},\quad \mathsf{X}+\mathsf{Z} \overset{\kappa_2}{\longrightarrow} 2\mathsf{Y},\quad \mathsf{Y} \overset{\kappa_3}{\longrightarrow} \mathsf{Z},\quad 2\mathsf{Z} \overset{\kappa_4}{\longrightarrow} 0\,,
\]
giving rise to the ODE
\[
\left(\begin{array}{c}\dot x \\\dot y \\ \dot z \end{array}\right) = \left(\begin{array}{rrrr}1&-1&0&0\\0&2&-1&0\\0&-1&1&-2\end{array}\right)\,\left(\begin{array}{c}\kappa_1 x\\\kappa_2 xz\\ \kappa_3 y\\ \kappa_4 z^2\end{array}\right)\,,
\]
where $x,y$ and $z$ are, respectively, the concentrations of $\mathsf{X}$, $\mathsf{Y}$ and $\mathsf{Z}$. This is network 23a in Appendix~\ref{appnetworks}. The  Hopf set in parameter space is given by:
\[
\mathrm{H} = \left\{(\alpha, \beta, \gamma) \in \mathbb{R}^3_{+}\,:\,\alpha = \beta(\beta+3\gamma)/(2\beta+3\gamma)\right\}\,,
\]
where $\alpha = \kappa_1$, $\beta = \kappa_3$, $\gamma = 2\kappa_1\kappa_4/\kappa_2$. In this network, the first Lyapunov coefficient, $L_1$, is able to take all signs on $\mathrm{H}$. We are able to find values of the rate constants at which there appear, in numerical simulations, to be both a stable periodic orbit and a stable equilibrium, as predicted by Theorem~\ref{thmBautin} (see Figure~\ref{figs3r4_23}). The rate constants are chosen following the theory as described above: beginning with a point on the Bautin set, we first move into a region where $L_1 > 0$, while remaining on the Hopf set; we then move off the Hopf set so as to create an unstable periodic orbit. In the simulations we find that orbits rapidly converge to a two dimensional surface; and then (generically) spiral slowly towards either the stable equilibrium or the stable periodic orbit. Between the two is an unstable periodic orbit, whose stable manifold appears to separate the basins of attraction of the equilibrium and the stable periodic orbit.

\begin{figure}[!h]
\centering
\includegraphics[scale=0.35]{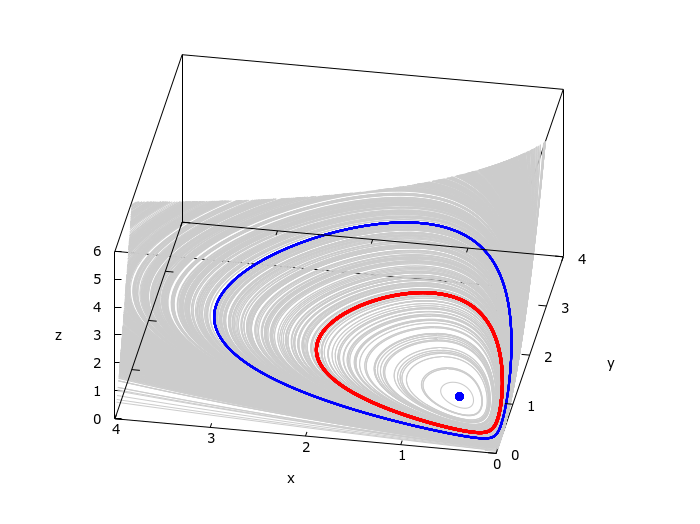}
\hspace{-1cm}\includegraphics[scale=0.35]{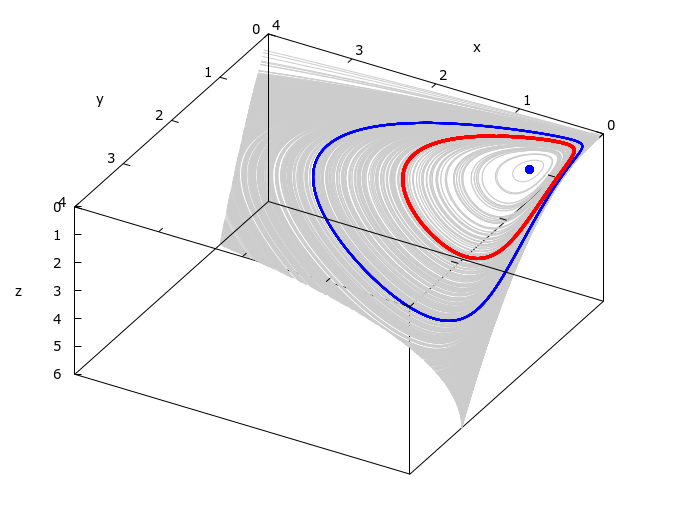}
\caption{Numerical simulations of network 23a. The same plot is shown from two different angles. We set rate constants $\kappa_1 = 0.65255, \, \kappa_2 = 1.0, \, \kappa_3 = 0.9$ and $\kappa_4 = 0.38312$. At these values of the rate constants the system appears to have two stable omega-limit sets: an equilibrium 
and a periodic orbit (shown in blue). 
Simulations show trajectories rapidly converging to what appears to be a two dimensional surface (shown in grey), and then slowly spiralling along the surface towards one of these sets. Apart from these two sets, the surface appears to include an unstable periodic orbit (shown in red).
}
\label{figs3r4_23}
\end{figure}

\end{example}

\section{Discussion and conclusions}
\label{secconclusions}

We remarked in the introduction that it was already observed by Wilhelm in \cite{Wilhelm2009} that $3$-species, $4$-reaction bimolecular networks are capable of Hopf bifurcation with mass action kinetics. In fact the network analysed by Wilhelm is the network numbered 40a in Appendix~\ref{appnetworks}. 

We were able to go further and fully classify the smallest bimolecular CRNs with the capacity for Hopf bifurcation, i.e., those with three species and four reactions. We found that, up to dynamical equivalence, 86 out of 6486 $(3,4,3)$ CRNs which admit nondegenerate equilibria, actually admit nondegenerate Hopf bifurcation. Moreover, $54$ of these networks admit a supercritical bifurcation resulting in the creation of a stable periodic orbit. 

We also find: (i) nondegenerate Bautin bifurcation in 25 networks; (ii) 31 networks which admit two nondegenerate periodic orbits; and (iii) 29 networks which admit a stable periodic orbit which coexists with a stable equilbrium. 

We find that easily checked necessary conditions for Hopf bifurcation completely determine the capacity of $3$-species, $4$-reaction, bimolecular, mass action networks for Hopf bifurcation, and almost completely determine their capacity for {\em nondegenerate} Hopf bifurcation. Up to dynamical equivalence, only a single exceptional network satisfies the necessary conditions, but fails conditions for nondegeneracy of the bifurcation. 

The exceptional network
\[
0 \rightarrow 2\mathsf{X}, \quad \mathsf{X}+\mathsf{Y} \rightarrow 2\mathsf{Y}, \quad \mathsf{Y}+\mathsf{Z} \rightarrow 2\mathsf{Z}, \quad \mathsf{X}+\mathsf{Z} \rightarrow 0\,
\]
is a nondegenerate network which robustly undergoes a degenerate Hopf bifurcation at all points on the bifurcation set. In fact, at bifurcation, the unique equilibrium is a center, and it can be proved that it has a unique, global center manifold: an unbounded two dimensional surface foliated by periodic orbits. This network thus admits a so-called vertical Hopf bifurcation. A detailed analysis of its dynamics can be found in \cite{BBHverticalHopf}. 

The work in this paper highlights why theoretical approaches are required for the study of oscillation in reaction networks. For example, stable oscillation in the ``fully open'' CRN
\[
\mathsf{X} + \mathsf{Z} \rightarrow 2\mathsf{Y}, \quad \mathsf{Y} + \mathsf{Z} \rightarrow 2\mathsf{Z}, \quad 0 \rightleftharpoons \mathsf{X}, \quad  0 \rightleftharpoons \mathsf{Y}, \quad 0 \rightleftharpoons \mathsf{Z}
\]
was found to occur in numerical simulations in \cite{banajiCRNosci}, and indeed {\em must} occur based on the theory in that paper because of the presence of network 11a in Appendix~\ref{appnetworks} as a subnetwork of maximal rank. The same theory implies that the fully open CRN
\[
\mathsf{X} + \mathsf{Z} \rightarrow \mathsf{Y} + \mathsf{Z} \rightarrow 2\mathsf{Z}, \quad  0 \rightleftharpoons \mathsf{X}, \quad  0 \rightleftharpoons \mathsf{Y}, \quad 0 \rightleftharpoons \mathsf{Z}
\]
must admit a nondegenerate, stable periodic orbit with mass action kinetics as it includes network 10a in Appendix~\ref{appnetworks} as a subnetwork of maximal rank. However, numerical simulations in \cite{banajiCRNosci} failed to identify this network as admitting stable oscillation with mass action kinetics, perhaps because stable oscillation occurs in a relatively small region of parameter space. 

A natural question is to what extent bimolecular $(3,4,3)$ mass action networks admitting Hopf bifurcation can be characterised by their network structure, described in graph theoretic terms. It is true that many of the networks capable of nondegenerate Hopf bifurcation presented in Appendix~\ref{appnetworks} have similar graphical representations, for example in terms of their Petri net graphs or DSR graphs. And it is known that such graphs can, indeed, encode important information about allowed dynamics \cite{craciunfeinbergSR,banajipetrinet,abphopf,baudier2018}. We present three examples to show, however, that any graph-theoretical characterisation of bimolecular $(3,4,3)$ mass action networks admitting Hopf bifurcation is likely to run into some complications. In each case, we present a pair of CRNs, where the second CRN differs from the first only in one stoichiometric coefficient in the product complex of one reaction. In each case both members of the pair are very similar in their various graphical representations, and in their associated differential equations; but the small change leads to significant changes in dynamical behaviour.

\begin{example}
\label{HopftoDynamicallyTrivial}
{\bf A small change makes the network dynamically trivial.} Consider the following pair of bimolecular $(3,4,3)$ networks:
\[
\mathsf{X} \rightarrow 2\mathsf{X}, \quad \mathsf{X}+\mathsf{Z} \rightarrow 2\mathsf{Y}, \quad \mathsf{X}+\mathsf{Y} \rightarrow \mathsf{Z}, \quad \mathsf{Z} \rightarrow 0\,
\]
and 
\[
\mathsf{X} \rightarrow 2\mathsf{X}, \quad \mathsf{X}+\mathsf{Z} \rightarrow \mathsf{Y}, \quad \mathsf{X}+\mathsf{Y} \rightarrow \mathsf{Z}, \quad \mathsf{Z} \rightarrow 0\,.
\]
The first appears as Network 21 in Appendix~\ref{appnetworks}, while the second is obtained from the first via a slight modification of the product complex of the second reaction. While the first admits a supercritical Hopf bifurcation, the second is dynamically trivial, and admits no positive limit sets at all.
\end{example}
While the situation in the previous example occurs frequently, it is also true that minor modifications can affect the capacity of a network for Hopf bifurcation, even if the network remains dynamically nontrivial. 
\begin{example} 
\label{HopftoNoHopf}
{\bf A small change leads to the loss of Hopf bifurcation.} Consider the following pair of bimolecular $(3,4,3)$ networks:
\[
\mathsf{X} \rightarrow 2\mathsf{X}, \quad 2\mathsf{X} \rightarrow 2\mathsf{Y}, \quad \mathsf{Y} \rightarrow 2\mathsf{Z}, \quad \mathsf{X}+\mathsf{Z} \rightarrow 0\,
\]
and
\[
\mathsf{X} \rightarrow 2\mathsf{X}, \quad 2\mathsf{X} \rightarrow 2\mathsf{Y}, \quad \mathsf{Y} \rightarrow \mathsf{Z}, \quad \mathsf{X}+\mathsf{Z} \rightarrow 0\,.
\]
The first appears as Network 26a in Appendix~\ref{appnetworks}, while the second is obtained from the first via a change to the product complex of the third reaction. The first network, with mass action kinetics, admits a supercritical Hopf bifurcation and, in fact, a nondegenerate Bautin bifurcation; while the second network admits no bifurcations of positive equilibria at all: the unique positive equilibrium of the second network is asymptotically stable for all choices of rate constants. 
\end{example}
In the opposite direction to the previous example, minor changes to a network can introduce more complex behaviours even amongst networks admitting Hopf bifurcation. 
\begin{example}
\label{NoBautintoBautin}
{\bf A small change introduces a codimension $2$ bifurcation.} Consider the following pair of bimolecular $(3,4,3)$ networks:
\[
0 \rightarrow \mathsf{X}, \quad \mathsf{X}+\mathsf{Y} \rightarrow 2\mathsf{Y}, \quad \mathsf{Y} \rightarrow \mathsf{Z}, \quad \mathsf{X}+\mathsf{Z} \rightarrow 0\,
\]
and
\[
0 \rightarrow \mathsf{X}, \quad \mathsf{X}+\mathsf{Y} \rightarrow 2\mathsf{Y}, \quad \mathsf{Y} \rightarrow 2\mathsf{Z}, \quad \mathsf{X}+\mathsf{Z} \rightarrow 0\,.
\]
These appear as Networks 3a and 6a in Appendix~\ref{appnetworks}. Both are capable of nondegenerate Hopf bifurcation. However, the second admits a nondegenerate Bautin bifurcation, while the first does not.
\end{example}

Finally, we can combine the approach taken in this paper with theory on the inheritance  of nondegenerate behaviours in CRNs summarised in \cite{banajisplitreacs} to make claims about larger networks with the potential for Hopf bifurcation. We mention just two of many possible results we can obtain from this process, and defer further analysis to future work. 
\begin{enumerate}[align=left,leftmargin=*]
\item Using the theory developed in \cite{bbhAMCrank} (see Remark~6 in \cite{bbhAMCrank}), we find that, up to isomorphism, 264 bimolecular $(4,4,3)$ CRNs, falling into 198 dynamically non-equivalent classes admit nondegenerate Hopf bifurcation simply because they include one of the bimolecular $(3,4,3)$ CRNs which admits nondegenerate Hopf bifurcation as a subnetwork in a natural sense. 

\item Using theory in \cite{banajiCRNosci} (see Remark~4.3 in \cite{banajiCRNosci}), we find that nondegenerate Hopf bifurcation can be predicted in $11192$ non-isomorphic bimolecular $(3,5,3)$ CRNs, as a consequence of a bimolecular $(3,4,3)$ subnetwork admitting nondegenerate Hopf bifurcation. These $(3,5,3)$ CRNs fall into $6176$ dynamically non-equivalent classes, of which $6129$ are not equivalent to any CRN with fewer reactions. 
\end{enumerate}
The code used to arrive at these numbers is available on GitHub \cite{muradgithub}.

\subsection*{Acknowledgements}
We would like to thank Josef Hofbauer, and the anonymous reviewers of this paper, for useful comments and suggestions which helped us to improve the paper.

\hspace{3cm}
\appendix
\section{Dynamical equivalence of dynamically nontrivial, bimolecular, $(n,n+1,n)$ networks}
\label{appdyniso}

Recall that we refer to two CRNs as {\bf dynamically equivalent} if we can relabel their species in such a way that they give rise to the same set of differential equations under the assumption of mass action kinetics. More precisely, consider two CRNs, say $\mathcal{R}_1$ and $\mathcal{R}_2$, on $n$ chemical species $\mathsf{X}_1, \ldots, \mathsf{X}_n$ and having, respectively $k_1$ and $k_2$ reactions. Suppose that they give rise to parameterised families of vector fields $f_1(x, \kappa)$ and $f_2(x, \nu)$ under the assumption of mass action kinetics. Here $x \in \mathbb{R}^n_{+}$ is the vector of species concentrations, and $\kappa \in \mathbb{R}^{k_1}_{+}$ and $\nu \in \mathbb{R}^{k_2}_{+}$ are the vectors of rate constants of $\mathcal{R}_1$ and $\mathcal{R}_2$ respectively. Then $\mathcal{R}_1$ and $\mathcal{R}_2$ are dynamically equivalent if, perhaps after permuting the species of one network,
\[
\{f_1(\cdot, \kappa)\colon \kappa \in \mathbb{R}^{k_1}_{+}\} = \{f_2(\cdot, \nu)\colon \nu \in \mathbb{R}^{k_2}_{+}\}\,.
\]
In other words,
\begin{enumerate}
\item for each $\tilde{\kappa} \in \mathbb{R}^{k_1}_{+}$, there exists $\tilde{\nu} \in \mathbb{R}^{k_2}_{+}$ such that $f_1(x, \tilde{\kappa}) = f_2(x, \tilde{\nu})$ for all $x \in \mathbb{R}^n_{+}$; and 
\item for each $\tilde{\nu} \in \mathbb{R}^{k_2}_{+}$, there exists $\tilde{\kappa} \in \mathbb{R}^{k_1}_{+}$ such that $f_1(x, \tilde{\kappa}) = f_2(x, \tilde{\nu})$ for all $x \in \mathbb{R}^n_{+}$.
\end{enumerate}

Dynamically equivalent CRNs are termed ``unconditionally confoundable'' in \cite{CraciunPanteaIdentifiability} where a necessary and sufficient condition for such equivalence is given. 

We wish to show that for the CRNs of interest to us in this paper, dynamical equivalence is itself equivalent to a simpler condition, which we term ``simple equivalence''. 

We will refer to two CRNs, say $\mathcal{R}_1$ and $\mathcal{R}_2$, as {\bf simply equivalent} if, perhaps after relabelling/reordering species and reactions, the following conditions hold:
\begin{enumerate}
\item They both have the same set of species, and the same number of reactions. 
\item The $i$th reactions of $\mathcal{R}_1$ and $\mathcal{R}_2$ have the same reactant complex.
\item The $i$th reaction vector of $\mathcal{R}_1$ is a positive multiple of the $i$th reaction vector of $\mathcal{R}_2$.  
\end{enumerate}

\enlargethispage{1\baselineskip}
We need two lemmas.

\begin{lemma}
\label{lemlindep}
In a dynamically nontrivial $(n,n+1,n)$ CRN, no set of $k$ reaction vectors can be linearly dependent for any $k \leq n$. 
\end{lemma}
\begin{proof}
Let $\Gamma$ be the stoichiometric matrix of an $(n,n+1,n)$ CRN. Suppose that reactions $i_1, \ldots, i_k$ are linearly dependent for some $k \leq n$. Then there exists a vector $v \in \mathrm{ker}\,\Gamma$ with support on this set, i.e., $v_j \neq 0$ if and only if $j \in \{i_1,\ldots, i_k\}$. But $\mathrm{ker}\,\Gamma$ is one dimensional, so all vectors in $\mathrm{ker}\,\Gamma$ are multiples of $v$. Hence $\mathrm{ker}\,\Gamma$ can include no positive vector, and the network cannot be dynamically nontrivial.
\end{proof}

\begin{lemma}
\label{lemtotRC}
Let $n \geq 2$, and let $\mathcal{R}$ be a dynamically nontrivial $(n, n+1, n)$ CRN. If $\mathcal{R}$ is nondegenerate, it must have $n+1$ distinct reactant complexes. If $\mathcal{R}$ is bimolecular, then it must have at least $2$ distinct reactant complexes.  
\end{lemma}
\begin{proof}
The first claim is an immediate consequence of the affine independence of the reactant complexes of a nondegenerate, dynamically nontrivial, $(n, n+1, n)$ network (see Remark~\ref{remND}).

Suppose $\mathcal{R}$ is bimolecular and has stoichiometric matrix $\Gamma$. Note that the {\em positive} span of the reaction vectors of a dynamically nontrivial $(n,n+1,n)$ CRN must be $\mathbb{R}^n$: we can solve any equation of the form $\Gamma \,v = w$ for $v$ as $\Gamma$ has rank $n$; and, moreover, we can assume without loss of generality that the solution $v$ is positive, by adding to $v$ if necessary an arbitrary multiple of some element of $\mathrm{ker}_{+}\,\Gamma$. On the other hand, all possible bimolecular complexes on $n \geq 2$ species lie on the boundary of their convex hull in $\mathbb{R}^n$, and so it is immediate that the positive span of any set of reaction vectors attached to one of these complexes, and terminating on these complexes, is at most a half-space in $\mathbb{R}^n$.
\end{proof}

We are now able to claim that for the networks we consider here, dynamical equivalence and simple equivalence coincide.
\begin{thm}
\label{thmdyniso}
Let $\mathcal{R}_1$ and $\mathcal{R}_2$ be two CRNs, and let $n \in \mathbb{N}$.
\begin{enumerate}
\item If $\mathcal{R}_1$ and $\mathcal{R}_2$ are simply equivalent, then they are dynamically equivalent.
\item If $\mathcal{R}_1$ and $\mathcal{R}_2$ are dynamically nontrivial, dynamically equivalent, $(n,n+1,n)$ CRNs which are either nondegenerate or bimolecular, then they are simply equivalent.
\end{enumerate}

\end{thm}
\begin{proof}
(1) Suppose $\mathcal{R}_1$ and $\mathcal{R}_2$ are simply equivalent. After renaming and reordering species and reactions if necessary, $\mathcal{R}_1$ and $\mathcal{R}_2$ have the same left stoichiometric matrix, say $\Gamma_l$. Moreover, by definition, their stoichiometric matrices, say $\Gamma_1$ and $\Gamma_2$, are related via $\Gamma_2 = \Gamma_1\Delta$ where $\Delta$ is a positive diagonal matrix. They thus give rise to sets of dynamical systems $\{\Gamma_1 (\kappa \circ x^{\Gamma_l^\mathrm{t}})\,:\, \kappa \in \mathbb{R}^m_{+}\}$ and $\{\Gamma_1\Delta (\kappa \circ x^{\Gamma_l^\mathrm{t}})\,:\,\kappa \in \mathbb{R}^m_{+}\}$ respectively, which are clearly identical. 

(2) Now fix $n \geq 1$ and suppose that $\mathcal{R}_1$ and $\mathcal{R}_2$ are dynamically nontrivial, nondegenerate, $(n,n+1,n)$ CRNs, which are also dynamically equivalent. We know from the proof of Theorem~4.4 in \cite{CraciunPanteaIdentifiability} that two CRNs are dynamically equivalent if and only if (perhaps after permuting species) they share the same reactant complexes and, for each reactant complex, the cones spanned by the reaction vectors associated with this complex are identical in both networks. We want to show that $\mathcal{R}_1$ and $\mathcal{R}_2$ are simply equivalent. 
 
The claim is trivial in the case $n=1$, so we assume $n \geq 2$. 
\begin{itemize}
\item[(i)] By Lemma~\ref{lemlindep}, in a dynamically nontrivial $(n,n+1,n)$ CRN, the only set of linearly dependent reaction vectors is the full set containing all $n+1$ reaction vectors.
\item[(ii)] By Lemma~\ref{lemtotRC}, CRNs satisfying the hypotheses of the theorem must have at least two reactant complexes. So, each reactant complex is the source for at most $n$ reactions in the network. 
\end{itemize}
Combining these two observations, the reaction vectors associated with any reactant complex must be linearly independent. They thus span a simplicial cone whose generators are, up to positive scaling, unique. Since $\mathcal{R}_1$ and $\mathcal{R}_2$ are dynamically equivalent, the sets of reaction vectors associated with any reactant complex must be identical in each network, up to positive scaling. This implies that $\mathcal{R}_1$ and $\mathcal{R}_2$ are simply equivalent.
\end{proof}

\subsection{Equivalent dynamics in CRNs which fail to be dynamically equivalent}
\label{secnondeq}
In the definition of dynamical equivalence, we do not allow transformations other than the permutation of species. It can occur, however, that two networks fail to be dynamically equivalent, but are nevertheless equivalent if we allow a wider range of transformations. Example~\ref{exnondeq} provides an instance of a pair of $(3,4,3)$ CRNs which fail to be dynamically equivalent, but nevertheless give rise to identical dynamics if we allow rescaling of species concentrations and reparameterisation.

\begin{example}
\label{exnondeq}
Consider the two CRNs
\[
\mathsf{X} \rightarrow 2\mathsf{X}, \quad \mathsf{X}+\mathsf{Z} \rightarrow \mathsf{Y}+\mathsf{Z}, \quad \mathsf{Y} \rightarrow 2\mathsf{Z}, \quad 2\mathsf{Z} \rightarrow \mathsf{0},
\]
and
\[
\mathsf{X} \rightarrow 2\mathsf{X}, \quad \mathsf{X}+\mathsf{Z} \rightarrow \mathsf{Y}+\mathsf{Z}, \quad \mathsf{Y} \rightarrow \mathsf{Z}, \quad 2\mathsf{Z} \rightarrow \mathsf{0}\,.
\]
These appear as networks 14a and 15a in Appendix~\ref{appnetworks}. If we carry out the recoordinatisation in Section~\ref{secrecoord}, the two networks in fact give rise to identical sets of differential equations. Thus the dynamics of the two networks is identical after a recoordinatisation and reparameterisation.
\end{example}

A list of equivalences of this kind which we were able to find amongst the bimolecular $3$-species, $4$-reaction networks admitting nondegenerate Hopf bifurcation, is given in Note~\ref{notenonsimpleequiv} in Appendix~\ref{appnetworks}.

\section{All bimolecular, $3$-species, $4$-reaction CRNs admitting nondegenerate Hopf bifurcation with mass action kinetics}
\label{appnetworks}

We provide a complete list of bimolecular $(3,4,3)$ networks admitting nondegenerate Hopf bifurcation with mass action kinetics. First, in Table~\ref{table86classify}, we revisit the classification of these networks according to reactant molecularity and the sign of the first Lyapunov coefficient (Table~\ref{table86summary}), but this time listing network numbers as they appear in the list to follow. 

\bgroup
\def\arraystretch{1.5}
\begin{table}[!htbp]
\begin{center}
\begin{tabular}{c"c|c|c|c|}
\multicolumn{1}{c}{}&\multicolumn{4}{c}{Reactant molecularity}\\
        & $(0,1,2,2)$ & $(1,1,2,2)$ & $(0,2,2,2)$ & $(1,2,2,2)$ \\
\thickhline
$L_1<0$ & $1-5$ & $14 - 21$ & $35 - 39$ & $40 - 53$ \\
\hline
$L_1 \gtreqless 0$ & $6$ & $22 - 27$ & $-$ & $54 - 71$ \\
\hline
$L_1>0$ & $7$ & $28 - 34$ & $-$ & $72 - 86$ \\
\hline
$L_1 \geq 0$ & $8 -  13$ & $-$ & $-$ & $-$ \\
\hline
\end{tabular}
\end{center}
\caption{\label{table86classify}Classification of the 86 non-equivalent bimolecular $(3,4,3)$ networks admitting nondegenerate Hopf bifurcation according to reactant molecularity and the possible signs of the first Lyapunov coefficient, $L_1$, at bifurcation (see Table~\ref{table86summary}). The numbers in each category refer to the CRN numbers in the list to follow.}
\end{table}
\egroup

In the following list, dynamically equivalent, but non-isomorphic, CRNs are given the same number, but differentiated by an additional letter (e.g., 1a and 1b).

\vspace{1cm}
\begin{center}
\begin{longtable}[!htbp]{|c@{\hspace{8pt}}lr@{\hspace{4pt}$\to$\hspace{4pt}}lr@{\hspace{4pt}$\to$\hspace{4pt}}lr@{\hspace{4pt}$\to$\hspace{4pt}}lr@{\hspace{4pt}$\to$\hspace{4pt}}l|}
\multicolumn{10}{c}{reactant molecularity $(0,1,2,2)$}\\
\cline{1-10}
\multirow{10}{*}{$L_1<0$} & 1a & $\mathsf{0}$ & $\mathsf{X}$ & $\mathsf{X}$ & $\mathsf{Y}$ & $\mathsf{Y}+\mathsf{Z}$ & $2\mathsf{Z}$ & $\mathsf{X}+\mathsf{Z}$ & $\mathsf{0}$ \\
& 1b & $\mathsf{0}$ & $2\mathsf{X}$ & $\mathsf{X}$ & $\mathsf{Y}$ & $\mathsf{Y}+\mathsf{Z}$ & $2\mathsf{Z}$ & $\mathsf{X}+\mathsf{Z}$ & $\mathsf{0}$ \\
& 2a & $\mathsf{0}$ & $\mathsf{X}$ & $\mathsf{X}$ & $2\mathsf{Y}$ & $\mathsf{Y}+\mathsf{Z}$ & $2\mathsf{Z}$ & $\mathsf{X}+\mathsf{Z}$ & $\mathsf{0}$ \\
& 2b & $\mathsf{0}$ & $2\mathsf{X}$ & $\mathsf{X}$ & $2\mathsf{Y}$ & $\mathsf{Y}+\mathsf{Z}$ & $2\mathsf{Z}$ & $\mathsf{X}+\mathsf{Z}$ & $\mathsf{0}$ \\
& 3a & $\mathsf{0}$ & $\mathsf{X}$ & $\mathsf{X}+\mathsf{Y}$ & $2\mathsf{Y}$ & $\mathsf{Y}$ & $\mathsf{Z}$ & $\mathsf{X}+\mathsf{Z}$ & $\mathsf{0}$ \\
& 3b & $\mathsf{0}$ & $2\mathsf{X}$ & $\mathsf{X}+\mathsf{Y}$ & $2\mathsf{Y}$ & $\mathsf{Y}$ & $\mathsf{Z}$ & $\mathsf{X}+\mathsf{Z}$ & $\mathsf{0}$ \\
& 4a & $\mathsf{0}$ & $\mathsf{X}$ & $\mathsf{X}+\mathsf{Y}$ & $2\mathsf{Y}$ & $\mathsf{Y}$ & $\mathsf{X}+\mathsf{Z}$ & $\mathsf{X}+\mathsf{Z}$ & $\mathsf{0}$ \\
& 4b & $\mathsf{0}$ & $2\mathsf{X}$ & $\mathsf{X}+\mathsf{Y}$ & $2\mathsf{Y}$ & $\mathsf{Y}$ & $\mathsf{X}+\mathsf{Z}$ & $\mathsf{X}+\mathsf{Z}$ & $\mathsf{0}$ \\
& 5a & $\mathsf{Z}$ & $\mathsf{X}+\mathsf{Z}$ & $\mathsf{X}+\mathsf{Y}$ & $2\mathsf{Y}$ & $\mathsf{Y}+\mathsf{Z}$ & $\mathsf{0}$ & $\mathsf{0}$ & $\mathsf{Z}$ \\
& 5b & $\mathsf{Z}$ & $\mathsf{X}+\mathsf{Z}$ & $\mathsf{X}+\mathsf{Y}$ & $2\mathsf{Y}$ & $\mathsf{Y}+\mathsf{Z}$ & $\mathsf{0}$ & $\mathsf{0}$ & $2\mathsf{Z}$ \\
\multicolumn{10}{|c|}{\cellcolor{Gray}}\\
\multirow{2}{*}{$L_1\gtreqless 0$}
& 6a & $\mathsf{0}$ & $\mathsf{X}$ & $\mathsf{X}+\mathsf{Y}$ & $2\mathsf{Y}$ & $\mathsf{Y}$ & $2\mathsf{Z}$ & $\mathsf{X}+\mathsf{Z}$ & $\mathsf{0}$ \\
& 6b & $\mathsf{0}$ & $2\mathsf{X}$ & $\mathsf{X}+\mathsf{Y}$ & $2\mathsf{Y}$ & $\mathsf{Y}$ & $2\mathsf{Z}$ & $\mathsf{X}+\mathsf{Z}$ & $\mathsf{0}$ \\
\multicolumn{10}{|c|}{\cellcolor{Gray}}\\
\multirow{1}{*}{$L_1>0$} & 7 & $\mathsf{0}$ & $\mathsf{X}+\mathsf{Y}$ & $\mathsf{X}+\mathsf{Z}$ & $\mathsf{Y}+\mathsf{Z}$ & $\mathsf{Y}+\mathsf{Z}$ & $2\mathsf{Z}$ & $\mathsf{Z}$ & $\mathsf{0}$ \\
\multicolumn{10}{|c|}{\cellcolor{Gray}}\\
\multirow{8}{*}{$L_1\geq0$} 
& 8 & $\mathsf{0}$ & $\mathsf{X}+\mathsf{Y}$ & $\mathsf{X}+\mathsf{Z}$ & $\mathsf{Y}$ & $\mathsf{Y}+\mathsf{Z}$ & $2\mathsf{Z}$ & $\mathsf{Z}$ & $\mathsf{0}$ \\
& 9 & $\mathsf{0}$ & $\mathsf{X}+\mathsf{Y}$ & $\mathsf{X}+\mathsf{Z}$ & $2\mathsf{Y}$ & $\mathsf{Y}+\mathsf{Z}$ & $2\mathsf{Z}$ & $\mathsf{Z}$ & $\mathsf{0}$ \\
& 10a & $\mathsf{0}$ & $\mathsf{X}$ & $\mathsf{X}+\mathsf{Z}$ & $\mathsf{Y}+\mathsf{Z}$ & $\mathsf{Y}+\mathsf{Z}$ & $2\mathsf{Z}$ & $\mathsf{Z}$ & $\mathsf{0}$ \\
& 10b & $\mathsf{0}$ & $2\mathsf{X}$ & $\mathsf{X}+\mathsf{Z}$ & $\mathsf{Y}+\mathsf{Z}$ & $\mathsf{Y}+\mathsf{Z}$ & $2\mathsf{Z}$ & $\mathsf{Z}$ & $\mathsf{0}$ \\
& 11a & $\mathsf{0}$ & $\mathsf{X}$ & $\mathsf{X}+\mathsf{Z}$ & $2\mathsf{Y}$ & $\mathsf{Y}+\mathsf{Z}$ & $2\mathsf{Z}$ & $\mathsf{Z}$ & $\mathsf{0}$ \\
& 11b & $\mathsf{0}$ & $2\mathsf{X}$ & $\mathsf{X}+\mathsf{Z}$ & $2\mathsf{Y}$ & $\mathsf{Y}+\mathsf{Z}$ & $2\mathsf{Z}$ & $\mathsf{Z}$ & $\mathsf{0}$ \\
& 12 & $\mathsf{0}$ & $\mathsf{X}+\mathsf{Z}$ & $\mathsf{X}+\mathsf{Y}$ & $2\mathsf{Y}$ & $\mathsf{Y}$ & $\mathsf{Z}$ & $\mathsf{Y}+\mathsf{Z}$ & $\mathsf{X}$ \\
& 13 & $\mathsf{0}$ & $\mathsf{X}+\mathsf{Z}$ & $\mathsf{X}+\mathsf{Y}$ & $2\mathsf{Y}$ & $\mathsf{Y}$ & $2\mathsf{Z}$ & $\mathsf{Y}+\mathsf{Z}$ & $\mathsf{X}$ \\
\cline{1-10}
\multicolumn{10}{c}{}\\
\multicolumn{10}{c}{}\\
\multicolumn{10}{c}{reactant molecularity $(1,1,2,2)$}\\
\cline{1-10}
\multirow{10}{*}{$L_1<0$}& 14a & $\mathsf{X}$ & $2\mathsf{X}$ & $\mathsf{X}+\mathsf{Z}$ & $\mathsf{Y}+\mathsf{Z}$ & $\mathsf{Y}$ & $2\mathsf{Z}$ & $2\mathsf{Z}$ & $\mathsf{0}$ \\
& 14b & $\mathsf{X}$ & $2\mathsf{X}$ & $\mathsf{X}+\mathsf{Z}$ & $\mathsf{Y}+\mathsf{Z}$ & $\mathsf{Y}$ & $2\mathsf{Z}$ & $2\mathsf{Z}$ & $\mathsf{Z}$ \\
& 15a & $\mathsf{X}$ & $2\mathsf{X}$ & $\mathsf{X}+\mathsf{Z}$ & $\mathsf{Y}+\mathsf{Z}$ & $\mathsf{Y}$ & $\mathsf{Z}$ & $2\mathsf{Z}$ & $\mathsf{0}$ \\
& 15b & $\mathsf{X}$ & $2\mathsf{X}$ & $\mathsf{X}+\mathsf{Z}$ & $\mathsf{Y}+\mathsf{Z}$ & $\mathsf{Y}$ & $\mathsf{Z}$ & $2\mathsf{Z}$ & $\mathsf{Z}$ \\
& 16 & $\mathsf{X}$ & $2\mathsf{X}$ & $\mathsf{X}+\mathsf{Z}$ & $\mathsf{Y}+\mathsf{Z}$ & $\mathsf{Y}$ & $\mathsf{Z}$ & $2\mathsf{Z}$ & $\mathsf{Y}$ \\
& 17 & $\mathsf{X}$ & $2\mathsf{X}$ & $\mathsf{X}+\mathsf{Z}$ & $\mathsf{Y}+\mathsf{Z}$ & $\mathsf{X}+\mathsf{Y}$ & $\mathsf{Z}$ & $\mathsf{Z}$ & $\mathsf{0}$ \\
& 18 & $\mathsf{X}$ & $2\mathsf{X}$ & $\mathsf{X}+\mathsf{Z}$ & $\mathsf{Y}+\mathsf{Z}$ & $\mathsf{X}+\mathsf{Y}$ & $2\mathsf{Z}$ & $\mathsf{Z}$ & $\mathsf{0}$ \\
& 19 & $\mathsf{X}$ & $2\mathsf{X}$ & $\mathsf{X}+\mathsf{Z}$ & $\mathsf{Y}$ & $\mathsf{X}+\mathsf{Y}$ & $2\mathsf{Z}$ & $\mathsf{Z}$ & $\mathsf{0}$ \\
& 20 & $\mathsf{X}$ & $2\mathsf{X}$ & $\mathsf{X}+\mathsf{Z}$ & $2\mathsf{Y}$ & $\mathsf{X}+\mathsf{Y}$ & $2\mathsf{Z}$ & $\mathsf{Z}$ & $\mathsf{0}$ \\
& 21 & $\mathsf{X}$ & $2\mathsf{X}$ & $\mathsf{X}+\mathsf{Z}$ & $2\mathsf{Y}$ & $\mathsf{X}+\mathsf{Y}$ & $\mathsf{Z}$ & $\mathsf{Z}$ & $\mathsf{0}$ \\
\multicolumn{10}{|c|}{\cellcolor{Gray}}\\
\multirow{10}{*}{$L_1\gtreqless 0$}
& 22 & $\mathsf{X}$ & $2\mathsf{X}$ & $\mathsf{X}+\mathsf{Z}$ & $2\mathsf{Y}$ & $\mathsf{Y}$ & $\mathsf{Z}$ & $2\mathsf{Z}$ & $\mathsf{Y}$ \\
& 23a & $\mathsf{X}$ & $2\mathsf{X}$ & $\mathsf{X}+\mathsf{Z}$ & $2\mathsf{Y}$ & $\mathsf{Y}$ & $\mathsf{Z}$ & $2\mathsf{Z}$ & $\mathsf{0}$ \\
& 23b & $\mathsf{X}$ & $2\mathsf{X}$ & $\mathsf{X}+\mathsf{Z}$ & $2\mathsf{Y}$ & $\mathsf{Y}$ & $\mathsf{Z}$ & $2\mathsf{Z}$ & $\mathsf{Z}$ \\
& 24a & $\mathsf{X}$ & $2\mathsf{X}$ & $\mathsf{X}+\mathsf{Z}$ & $2\mathsf{Y}$ & $\mathsf{Y}$ & $2\mathsf{Z}$ & $2\mathsf{Z}$ & $\mathsf{0}$ \\
& 24b & $\mathsf{X}$ & $2\mathsf{X}$ & $\mathsf{X}+\mathsf{Z}$ & $2\mathsf{Y}$ & $\mathsf{Y}$ & $2\mathsf{Z}$ & $2\mathsf{Z}$ & $\mathsf{Z}$ \\
& 25a & $\mathsf{X}$ & $2\mathsf{X}$ & $\mathsf{X}+\mathsf{Z}$ & $\mathsf{Y}$ & $\mathsf{Y}$ & $2\mathsf{Z}$ & $2\mathsf{Z}$ & $\mathsf{0}$ \\
& 25b & $\mathsf{X}$ & $2\mathsf{X}$ & $\mathsf{X}+\mathsf{Z}$ & $\mathsf{Y}$ & $\mathsf{Y}$ & $2\mathsf{Z}$ & $2\mathsf{Z}$ & $\mathsf{Z}$ \\
& 26a & $\mathsf{X}$ & $2\mathsf{X}$ & $2\mathsf{X}$ & $2\mathsf{Y}$ & $\mathsf{Y}$ & $2\mathsf{Z}$ & $\mathsf{X}+\mathsf{Z}$ & $\mathsf{0}$ \\
& 26b & $\mathsf{X}$ & $2\mathsf{X}$ & $2\mathsf{X}$ & $\mathsf{X}+\mathsf{Y}$ & $\mathsf{Y}$ & $2\mathsf{Z}$ & $\mathsf{X}+\mathsf{Z}$ & $\mathsf{0}$ \\
& 27a & $\mathsf{Z}$ & $2\mathsf{X}$ & $\mathsf{X}+\mathsf{Y}$ & $2\mathsf{Y}$ & $\mathsf{Y}$ & $\mathsf{0}$ & $2\mathsf{X}$ & $2\mathsf{Z}$ \\
& 27b & $\mathsf{Z}$ & $2\mathsf{X}$ & $\mathsf{X}+\mathsf{Y}$ & $2\mathsf{Y}$ & $\mathsf{Y}$ & $\mathsf{0}$ & $2\mathsf{X}$ & $\mathsf{X}+\mathsf{Z}$ \\
\multicolumn{10}{|c|}{\cellcolor{Gray}}\\
\multirow{9}{*}{$L_1>0$} 
& 28 & $\mathsf{Y}$ & $2\mathsf{X}$ & $\mathsf{X}+\mathsf{Z}$ & $\mathsf{Y}+\mathsf{Z}$ & $2\mathsf{Y}$ & $\mathsf{Z}$ & $\mathsf{Z}$ & $\mathsf{0}$ \\
& 29 & $\mathsf{Y}$ & $2\mathsf{X}$ & $\mathsf{X}+\mathsf{Z}$ & $\mathsf{Y}+\mathsf{Z}$ & $2\mathsf{Y}$ & $\mathsf{Z}$ & $\mathsf{Z}$ & $\mathsf{X}$ \\
& 30 & $\mathsf{Y}$ & $2\mathsf{X}$ & $\mathsf{X}+\mathsf{Z}$ & $\mathsf{Y}+\mathsf{Z}$ & $2\mathsf{Y}$ & $\mathsf{Z}$ & $\mathsf{Z}$ & $\mathsf{Y}$ \\
& 31 & $\mathsf{Y}$ & $2\mathsf{X}$ & $\mathsf{X}+\mathsf{Z}$ & $\mathsf{Y}+\mathsf{Z}$ & $2\mathsf{Y}$ & $\mathsf{X}+\mathsf{Z}$ & $\mathsf{Z}$ & $\mathsf{0}$ \\
& 32a & $\mathsf{Y}$ & $2\mathsf{X}$ & $\mathsf{X}+\mathsf{Z}$ & $\mathsf{Y}+\mathsf{Z}$ & $2\mathsf{Y}$ & $2\mathsf{Z}$ & $\mathsf{Z}$ & $\mathsf{0}$ \\
& 32b & $\mathsf{Y}$ & $2\mathsf{X}$ & $\mathsf{X}+\mathsf{Z}$ & $\mathsf{Y}+\mathsf{Z}$ & $2\mathsf{Y}$ & $\mathsf{Y}+\mathsf{Z}$ & $\mathsf{Z}$ & $\mathsf{0}$ \\
& 33a & $\mathsf{Y}$ & $2\mathsf{X}$ & $\mathsf{X}+\mathsf{Z}$ & $2\mathsf{Y}$ & $2\mathsf{Y}$ & $2\mathsf{Z}$ & $\mathsf{Z}$ & $\mathsf{0}$ \\
& 33b & $\mathsf{Y}$ & $2\mathsf{X}$ & $\mathsf{X}+\mathsf{Z}$ & $2\mathsf{Y}$ & $2\mathsf{Y}$ & $\mathsf{Y}+\mathsf{Z}$ & $\mathsf{Z}$ & $\mathsf{0}$ \\
& 34 & $\mathsf{Z}$ & $2\mathsf{X}$ & $\mathsf{X}+\mathsf{Y}$ & $2\mathsf{Y}$ & $\mathsf{Y}$ & $\mathsf{Z}$ & $2\mathsf{Z}$ & $\mathsf{Y}$ \\
\cline{1-10}
\multicolumn{10}{c}{}\\
\multicolumn{10}{c}{}\\
\multicolumn{10}{c}{reactant molecularity $(0,2,2,2)$}\\
\cline{1-10}
\multirow{10}{*}{$L_1<0$}
& 35a & $\mathsf{0}$ & $\mathsf{X}$ & $2\mathsf{X}$ & $\mathsf{Y}$ & $\mathsf{Y}+\mathsf{Z}$ & $2\mathsf{Z}$ & $\mathsf{X}+\mathsf{Z}$ & $\mathsf{0}$ \\
& 35b & $\mathsf{0}$ & $2\mathsf{X}$ & $2\mathsf{X}$ & $\mathsf{Y}$ & $\mathsf{Y}+\mathsf{Z}$ & $2\mathsf{Z}$ & $\mathsf{X}+\mathsf{Z}$ & $\mathsf{0}$ \\
& 36a & $\mathsf{0}$ & $\mathsf{X}$ & $2\mathsf{X}$ & $2\mathsf{Y}$ & $\mathsf{Y}+\mathsf{Z}$ & $2\mathsf{Z}$ & $\mathsf{X}+\mathsf{Z}$ & $\mathsf{0}$ \\
& 36b & $\mathsf{0}$ & $\mathsf{X}$ & $2\mathsf{X}$ & $\mathsf{X}+\mathsf{Y}$ & $\mathsf{Y}+\mathsf{Z}$ & $2\mathsf{Z}$ & $\mathsf{X}+\mathsf{Z}$ & $\mathsf{0}$ \\
& 36c & $\mathsf{0}$ & $2\mathsf{X}$ & $2\mathsf{X}$ & $2\mathsf{Y}$ & $\mathsf{Y}+\mathsf{Z}$ & $2\mathsf{Z}$ & $\mathsf{X}+\mathsf{Z}$ & $\mathsf{0}$ \\
& 36d & $\mathsf{0}$ & $2\mathsf{X}$ & $2\mathsf{X}$ & $\mathsf{X}+\mathsf{Y}$ & $\mathsf{Y}+\mathsf{Z}$ & $2\mathsf{Z}$ & $\mathsf{X}+\mathsf{Z}$ & $\mathsf{0}$ \\
& 37a & $\mathsf{0}$ & $\mathsf{X}$ & $\mathsf{X}+\mathsf{Y}$ & $2\mathsf{Y}$ & $\mathsf{Y}+\mathsf{Z}$ & $2\mathsf{Z}$ & $\mathsf{X}+\mathsf{Z}$ & $\mathsf{X}$ \\
& 37b & $\mathsf{0}$ & $2\mathsf{X}$ & $\mathsf{X}+\mathsf{Y}$ & $2\mathsf{Y}$ & $\mathsf{Y}+\mathsf{Z}$ & $2\mathsf{Z}$ & $\mathsf{X}+\mathsf{Z}$ & $\mathsf{X}$ \\
& 38 & $\mathsf{0}$ & $\mathsf{X}+\mathsf{Y}$ & $\mathsf{X}+\mathsf{Y}$ & $2\mathsf{Y}$ & $\mathsf{Y}+\mathsf{Z}$ & $2\mathsf{Z}$ & $\mathsf{X}+\mathsf{Z}$ & $\mathsf{X}$ \\
& 39 & $\mathsf{0}$ & $\mathsf{X}+\mathsf{Z}$ & $\mathsf{X}+\mathsf{Y}$ & $2\mathsf{Y}$ & $\mathsf{Y}+\mathsf{Z}$ & $\mathsf{Z}$ & $\mathsf{X}+\mathsf{Z}$ & $\mathsf{X}$ \\
\cline{1-10}
\multicolumn{10}{c}{}\\
\multicolumn{10}{c}{}\\
\multicolumn{10}{c}{reactant molecularity $(1,2,2,2)$}\\
\cline{1-10}
\multirow{24}{*}{$L_1<0$}
& 40a & $\mathsf{X}$ & $2\mathsf{X}$ & $\mathsf{X}+\mathsf{Z}$ & $\mathsf{Y}+\mathsf{Z}$ & $2\mathsf{Y}$ & $\mathsf{Z}$ & $2\mathsf{Z}$ & $\mathsf{0}$ \\
& 40b & $\mathsf{X}$ & $2\mathsf{X}$ & $\mathsf{X}+\mathsf{Z}$ & $\mathsf{Y}+\mathsf{Z}$ & $2\mathsf{Y}$ & $\mathsf{Z}$ & $2\mathsf{Z}$ & $\mathsf{Z}$ \\
& 41 & $\mathsf{X}$ & $2\mathsf{X}$ & $\mathsf{X}+\mathsf{Z}$ & $\mathsf{Y}+\mathsf{Z}$ & $2\mathsf{Y}$ & $\mathsf{Z}$ & $2\mathsf{Z}$ & $\mathsf{Y}$ \\
& 42a & $\mathsf{X}$ & $2\mathsf{X}$ & $\mathsf{X}+\mathsf{Z}$ & $\mathsf{Y}+\mathsf{Z}$ & $2\mathsf{Y}$ & $\mathsf{Z}$ & $2\mathsf{Z}$ & $2\mathsf{Y}$ \\
& 42b & $\mathsf{X}$ & $2\mathsf{X}$ & $\mathsf{X}+\mathsf{Z}$ & $\mathsf{Y}+\mathsf{Z}$ & $2\mathsf{Y}$ & $\mathsf{Z}$ & $2\mathsf{Z}$ & $\mathsf{Y}+\mathsf{Z}$ \\
& 43 & $\mathsf{X}$ & $2\mathsf{X}$ & $\mathsf{X}+\mathsf{Z}$ & $\mathsf{Y}+\mathsf{Z}$ & $2\mathsf{Y}$ & $\mathsf{Z}$ & $\mathsf{Y}+\mathsf{Z}$ & $\mathsf{0}$ \\
& 44 & $\mathsf{X}$ & $2\mathsf{X}$ & $\mathsf{X}+\mathsf{Z}$ & $\mathsf{Y}+\mathsf{Z}$ & $2\mathsf{Y}$ & $\mathsf{Z}$ & $\mathsf{Y}+\mathsf{Z}$ & $\mathsf{Y}$ \\
& 45 & $\mathsf{X}$ & $2\mathsf{X}$ & $\mathsf{X}+\mathsf{Z}$ & $\mathsf{Y}+\mathsf{Z}$ & $2\mathsf{Y}$ & $\mathsf{Z}$ & $\mathsf{Y}+\mathsf{Z}$ & $2\mathsf{Y}$ \\
& 46a & $\mathsf{X}$ & $2\mathsf{X}$ & $\mathsf{X}+\mathsf{Z}$ & $\mathsf{Y}+\mathsf{Z}$ & $2\mathsf{Y}$ & $2\mathsf{Z}$ & $2\mathsf{Z}$ & $\mathsf{0}$ \\
& 46b & $\mathsf{X}$ & $2\mathsf{X}$ & $\mathsf{X}+\mathsf{Z}$ & $\mathsf{Y}+\mathsf{Z}$ & $2\mathsf{Y}$ & $2\mathsf{Z}$ & $2\mathsf{Z}$ & $\mathsf{Z}$ \\
& 46c & $\mathsf{X}$ & $2\mathsf{X}$ & $\mathsf{X}+\mathsf{Z}$ & $\mathsf{Y}+\mathsf{Z}$ & $2\mathsf{Y}$ & $\mathsf{Y}+\mathsf{Z}$ & $2\mathsf{Z}$ & $\mathsf{0}$ \\
& 46d & $\mathsf{X}$ & $2\mathsf{X}$ & $\mathsf{X}+\mathsf{Z}$ & $\mathsf{Y}+\mathsf{Z}$ & $2\mathsf{Y}$ & $\mathsf{Y}+\mathsf{Z}$ & $2\mathsf{Z}$ & $\mathsf{Z}$ \\
& 47a & $\mathsf{X}$ & $2\mathsf{X}$ & $\mathsf{X}+\mathsf{Z}$ & $\mathsf{Y}+\mathsf{Z}$ & $2\mathsf{Y}$ & $2\mathsf{Z}$ & $2\mathsf{Z}$ & $\mathsf{Y}$ \\
& 47b & $\mathsf{X}$ & $2\mathsf{X}$ & $\mathsf{X}+\mathsf{Z}$ & $\mathsf{Y}+\mathsf{Z}$ & $2\mathsf{Y}$ & $\mathsf{Y}+\mathsf{Z}$ & $2\mathsf{Z}$ & $\mathsf{Y}$ \\
& 48a & $\mathsf{X}$ & $2\mathsf{X}$ & $\mathsf{X}+\mathsf{Z}$ & $\mathsf{Y}+\mathsf{Z}$ & $2\mathsf{Y}$ & $2\mathsf{Z}$ & $\mathsf{Y}+\mathsf{Z}$ & $\mathsf{0}$ \\
& 48b & $\mathsf{X}$ & $2\mathsf{X}$ & $\mathsf{X}+\mathsf{Z}$ & $\mathsf{Y}+\mathsf{Z}$ & $2\mathsf{Y}$ & $\mathsf{Y}+\mathsf{Z}$ & $\mathsf{Y}+\mathsf{Z}$ & $\mathsf{0}$ \\
& 49a & $\mathsf{X}$ & $2\mathsf{X}$ & $\mathsf{X}+\mathsf{Z}$ & $\mathsf{Y}+\mathsf{Z}$ & $2\mathsf{Y}$ & $2\mathsf{Z}$ & $\mathsf{Y}+\mathsf{Z}$ & $\mathsf{Y}$ \\
& 49b & $\mathsf{X}$ & $2\mathsf{X}$ & $\mathsf{X}+\mathsf{Z}$ & $\mathsf{Y}+\mathsf{Z}$ & $2\mathsf{Y}$ & $\mathsf{Y}+\mathsf{Z}$ & $\mathsf{Y}+\mathsf{Z}$ & $\mathsf{Y}$ \\
& 50a & $\mathsf{X}$ & $2\mathsf{X}$ & $\mathsf{X}+\mathsf{Z}$ & $\mathsf{Y}+\mathsf{Z}$ & $\mathsf{Y}+\mathsf{Z}$ & $2\mathsf{Z}$ & $2\mathsf{Z}$ & $\mathsf{0}$ \\
& 50b & $\mathsf{X}$ & $2\mathsf{X}$ & $\mathsf{X}+\mathsf{Z}$ & $\mathsf{Y}+\mathsf{Z}$ & $\mathsf{Y}+\mathsf{Z}$ & $2\mathsf{Z}$ & $2\mathsf{Z}$ & $\mathsf{Z}$ \\
& 51 & $\mathsf{X}$ & $2\mathsf{X}$ & $\mathsf{X}+\mathsf{Z}$ & $\mathsf{Y}+\mathsf{Z}$ & $\mathsf{Y}+\mathsf{Z}$ & $2\mathsf{Z}$ & $2\mathsf{Z}$ & $\mathsf{Y}$ \\
& 52 & $\mathsf{X}$ & $2\mathsf{X}$ & $2\mathsf{X}$ & $\mathsf{Y}$ & $\mathsf{Y}+\mathsf{Z}$ & $2\mathsf{Z}$ & $\mathsf{X}+\mathsf{Z}$ & $\mathsf{0}$ \\
& 53a & $\mathsf{X}$ & $2\mathsf{X}$ & $2\mathsf{X}$ & $2\mathsf{Y}$ & $\mathsf{Y}+\mathsf{Z}$ & $2\mathsf{Z}$ & $\mathsf{X}+\mathsf{Z}$ & $\mathsf{0}$ \\
& 53b & $\mathsf{X}$ & $2\mathsf{X}$ & $2\mathsf{X}$ & $\mathsf{X}+\mathsf{Y}$ & $\mathsf{Y}+\mathsf{Z}$ & $2\mathsf{Z}$ & $\mathsf{X}+\mathsf{Z}$ & $\mathsf{0}$ \\
\multicolumn{10}{|c|}{\cellcolor{Gray}}\\
 \multirow{29}{*}{$ L_1\gtreqless 0$}
& 54 & $\mathsf{X}$ & $2\mathsf{X}$ & $\mathsf{X}+\mathsf{Y}$ & $2\mathsf{Y}$ & $2\mathsf{Y}$ & $\mathsf{Z}$ & $\mathsf{X}+\mathsf{Z}$ & $\mathsf{Y}$ \\
& 55a & $\mathsf{X}$ & $2\mathsf{X}$ & $\mathsf{X}+\mathsf{Z}$ & $\mathsf{Y}$ & $\mathsf{X}+\mathsf{Y}$ & $2\mathsf{Z}$ & $2\mathsf{Z}$ & $\mathsf{0}$ \\
& 55b & $\mathsf{X}$ & $2\mathsf{X}$ & $\mathsf{X}+\mathsf{Z}$ & $\mathsf{Y}$ & $\mathsf{X}+\mathsf{Y}$ & $2\mathsf{Z}$ & $2\mathsf{Z}$ & $\mathsf{Z}$ \\
& 56 & $\mathsf{X}$ & $2\mathsf{X}$ & $\mathsf{X}+\mathsf{Z}$ & $\mathsf{Y}$ & $\mathsf{X}+\mathsf{Y}$ & $2\mathsf{Z}$ & $\mathsf{Y}+\mathsf{Z}$ & $\mathsf{Y}$ \\
& 57a & $\mathsf{X}$ & $2\mathsf{X}$ & $\mathsf{X}+\mathsf{Z}$ & $2\mathsf{Y}$ & $2\mathsf{Y}$ & $2\mathsf{Z}$ & $2\mathsf{Z}$ & $\mathsf{0}$ \\
& 57b & $\mathsf{X}$ & $2\mathsf{X}$ & $\mathsf{X}+\mathsf{Z}$ & $2\mathsf{Y}$ & $2\mathsf{Y}$ & $2\mathsf{Z}$ & $2\mathsf{Z}$ & $\mathsf{Z}$ \\
& 57c & $\mathsf{X}$ & $2\mathsf{X}$ & $\mathsf{X}+\mathsf{Z}$ & $2\mathsf{Y}$ & $2\mathsf{Y}$ & $\mathsf{Y}+\mathsf{Z}$ & $2\mathsf{Z}$ & $\mathsf{0}$ \\
& 57d & $\mathsf{X}$ & $2\mathsf{X}$ & $\mathsf{X}+\mathsf{Z}$ & $2\mathsf{Y}$ & $2\mathsf{Y}$ & $\mathsf{Y}+\mathsf{Z}$ & $2\mathsf{Z}$ & $\mathsf{Z}$ \\
& 58a & $\mathsf{X}$ & $2\mathsf{X}$ & $\mathsf{X}+\mathsf{Z}$ & $2\mathsf{Y}$ & $2\mathsf{Y}$ & $2\mathsf{Z}$ & $2\mathsf{Z}$ & $\mathsf{Y}$ \\
& 58b & $\mathsf{X}$ & $2\mathsf{X}$ & $\mathsf{X}+\mathsf{Z}$ & $2\mathsf{Y}$ & $2\mathsf{Y}$ & $\mathsf{Y}+\mathsf{Z}$ & $2\mathsf{Z}$ & $\mathsf{Y}$ \\
& 59a & $\mathsf{X}$ & $2\mathsf{X}$ & $\mathsf{X}+\mathsf{Z}$ & $2\mathsf{Y}$ & $2\mathsf{Y}$ & $2\mathsf{Z}$ & $\mathsf{Y}+\mathsf{Z}$ & $\mathsf{0}$ \\
& 59b & $\mathsf{X}$ & $2\mathsf{X}$ & $\mathsf{X}+\mathsf{Z}$ & $2\mathsf{Y}$ & $2\mathsf{Y}$ & $\mathsf{Y}+\mathsf{Z}$ & $\mathsf{Y}+\mathsf{Z}$ & $\mathsf{0}$ \\
& 60a & $\mathsf{X}$ & $2\mathsf{X}$ & $\mathsf{X}+\mathsf{Z}$ & $2\mathsf{Y}$ & $2\mathsf{Y}$ & $2\mathsf{Z}$ & $\mathsf{Y}+\mathsf{Z}$ & $\mathsf{Y}$ \\
& 60b & $\mathsf{X}$ & $2\mathsf{X}$ & $\mathsf{X}+\mathsf{Z}$ & $2\mathsf{Y}$ & $2\mathsf{Y}$ & $\mathsf{Y}+\mathsf{Z}$ & $\mathsf{Y}+\mathsf{Z}$ & $\mathsf{Y}$ \\
& 61a & $\mathsf{X}$ & $2\mathsf{X}$ & $\mathsf{X}+\mathsf{Z}$ & $2\mathsf{Y}$ & $\mathsf{X}+\mathsf{Y}$ & $\mathsf{Z}$ & $2\mathsf{Z}$ & $\mathsf{0}$ \\
& 61b & $\mathsf{X}$ & $2\mathsf{X}$ & $\mathsf{X}+\mathsf{Z}$ & $2\mathsf{Y}$ & $\mathsf{X}+\mathsf{Y}$ & $\mathsf{Z}$ & $2\mathsf{Z}$ & $\mathsf{Z}$ \\
& 62 & $\mathsf{X}$ & $2\mathsf{X}$ & $\mathsf{X}+\mathsf{Z}$ & $2\mathsf{Y}$ & $\mathsf{X}+\mathsf{Y}$ & $\mathsf{Z}$ & $2\mathsf{Z}$ & $\mathsf{Y}$ \\
& 63 & $\mathsf{X}$ & $2\mathsf{X}$ & $\mathsf{X}+\mathsf{Z}$ & $2\mathsf{Y}$ & $\mathsf{X}+\mathsf{Y}$ & $\mathsf{Z}$ & $\mathsf{Y}+\mathsf{Z}$ & $\mathsf{Y}$ \\
& 64a & $\mathsf{X}$ & $2\mathsf{X}$ & $\mathsf{X}+\mathsf{Z}$ & $2\mathsf{Y}$ & $\mathsf{X}+\mathsf{Y}$ & $2\mathsf{Z}$ & $2\mathsf{Z}$ & $\mathsf{0}$ \\
& 64b & $\mathsf{X}$ & $2\mathsf{X}$ & $\mathsf{X}+\mathsf{Z}$ & $2\mathsf{Y}$ & $\mathsf{X}+\mathsf{Y}$ & $2\mathsf{Z}$ & $2\mathsf{Z}$ & $\mathsf{Z}$ \\
& 65 & $\mathsf{X}$ & $2\mathsf{X}$ & $\mathsf{X}+\mathsf{Z}$ & $2\mathsf{Y}$ & $\mathsf{X}+\mathsf{Y}$ & $2\mathsf{Z}$ & $\mathsf{Y}+\mathsf{Z}$ & $\mathsf{Y}$ \\
& 66 & $\mathsf{X}$ & $2\mathsf{X}$ & $\mathsf{X}+\mathsf{Z}$ & $2\mathsf{Y}$ & $\mathsf{X}+\mathsf{Y}$ & $\mathsf{X}+\mathsf{Z}$ & $\mathsf{Y}+\mathsf{Z}$ & $\mathsf{0}$ \\
& 67a & $\mathsf{X}$ & $2\mathsf{X}$ & $\mathsf{X}+\mathsf{Z}$ & $\mathsf{Y}+\mathsf{Z}$ & $\mathsf{X}+\mathsf{Y}$ & $\mathsf{Z}$ & $2\mathsf{Z}$ & $\mathsf{0}$ \\
& 67b & $\mathsf{X}$ & $2\mathsf{X}$ & $\mathsf{X}+\mathsf{Z}$ & $\mathsf{Y}+\mathsf{Z}$ & $\mathsf{X}+\mathsf{Y}$ & $\mathsf{Z}$ & $2\mathsf{Z}$ & $\mathsf{Z}$ \\
& 68 & $\mathsf{X}$ & $2\mathsf{X}$ & $\mathsf{X}+\mathsf{Z}$ & $\mathsf{Y}+\mathsf{Z}$ & $\mathsf{X}+\mathsf{Y}$ & $\mathsf{Z}$ & $2\mathsf{Z}$ & $\mathsf{Y}$ \\
& 69 & $\mathsf{X}$ & $2\mathsf{X}$ & $\mathsf{X}+\mathsf{Z}$ & $\mathsf{Y}+\mathsf{Z}$ & $\mathsf{X}+\mathsf{Y}$ & $\mathsf{Z}$ & $\mathsf{Y}+\mathsf{Z}$ & $\mathsf{Y}$ \\
& 70a & $\mathsf{X}$ & $2\mathsf{X}$ & $\mathsf{X}+\mathsf{Z}$ & $\mathsf{Y}+\mathsf{Z}$ & $\mathsf{X}+\mathsf{Y}$ & $2\mathsf{Z}$ & $2\mathsf{Z}$ & $\mathsf{0}$ \\
& 70b & $\mathsf{X}$ & $2\mathsf{X}$ & $\mathsf{X}+\mathsf{Z}$ & $\mathsf{Y}+\mathsf{Z}$ & $\mathsf{X}+\mathsf{Y}$ & $2\mathsf{Z}$ & $2\mathsf{Z}$ & $\mathsf{Z}$ \\
& 71 & $\mathsf{X}$ & $2\mathsf{X}$ & $\mathsf{X}+\mathsf{Z}$ & $\mathsf{Y}+\mathsf{Z}$ & $\mathsf{X}+\mathsf{Y}$ & $2\mathsf{Z}$ & $\mathsf{Y}+\mathsf{Z}$ & $\mathsf{Y}$ \\
\multicolumn{10}{|c|}{\cellcolor{Gray}}\\
\multirow{15}{*}{$L_1 > 0$} & 72a & $\mathsf{X}$ & $2\mathsf{X}$ & $\mathsf{X}+\mathsf{Z}$ & $2\mathsf{Y}$ & $2\mathsf{Y}$ & $\mathsf{0}$ & $\mathsf{X}+\mathsf{Y}$ & $\mathsf{X}+\mathsf{Z}$ \\
& 72b & $\mathsf{X}$ & $2\mathsf{X}$ & $\mathsf{X}+\mathsf{Z}$ & $2\mathsf{Y}$ & $2\mathsf{Y}$ & $\mathsf{Y}$ & $\mathsf{X}+\mathsf{Y}$ & $\mathsf{X}+\mathsf{Z}$ \\
& 73 & $\mathsf{X}$ & $2\mathsf{X}$ & $\mathsf{X}+\mathsf{Z}$ & $2\mathsf{Y}$ & $\mathsf{Y}+\mathsf{Z}$ & $\mathsf{Z}$ & $\mathsf{X}+\mathsf{Y}$ & $\mathsf{X}+\mathsf{Z}$ \\
& 74a & $\mathsf{X}$ & $2\mathsf{X}$ & $\mathsf{X}+\mathsf{Z}$ & $2\mathsf{Y}$ & $\mathsf{Y}+\mathsf{Z}$ & $\mathsf{Z}$ & $2\mathsf{Y}$ & $2\mathsf{Z}$ \\
& 74b & $\mathsf{X}$ & $2\mathsf{X}$ & $\mathsf{X}+\mathsf{Z}$ & $2\mathsf{Y}$ & $\mathsf{Y}+\mathsf{Z}$ & $\mathsf{Z}$ & $2\mathsf{Y}$ & $\mathsf{Y}+\mathsf{Z}$ \\
& 75a & $\mathsf{X}$ & $2\mathsf{X}$ & $\mathsf{X}+\mathsf{Z}$ & $2\mathsf{Y}$ & $\mathsf{Y}+\mathsf{Z}$ & $2\mathsf{Z}$ & $2\mathsf{Z}$ & $\mathsf{0}$ \\
& 75b & $\mathsf{X}$ & $2\mathsf{X}$ & $\mathsf{X}+\mathsf{Z}$ & $2\mathsf{Y}$ & $\mathsf{Y}+\mathsf{Z}$ & $2\mathsf{Z}$ & $2\mathsf{Z}$ & $\mathsf{Z}$ \\
& 76 & $\mathsf{X}$ & $2\mathsf{X}$ & $\mathsf{X}+\mathsf{Z}$ & $2\mathsf{Y}$ & $\mathsf{Y}+\mathsf{Z}$ & $2\mathsf{Z}$ & $2\mathsf{Z}$ & $\mathsf{Y}$ \\
& 77 & $\mathsf{X}$ & $2\mathsf{X}$ & $\mathsf{X}+\mathsf{Z}$ & $\mathsf{Y}+\mathsf{Z}$ & $2\mathsf{Y}$ & $\mathsf{X}+\mathsf{Z}$ & $\mathsf{Y}+\mathsf{Z}$ & $\mathsf{0}$ \\
& 78 & $\mathsf{X}$ & $2\mathsf{X}$ & $\mathsf{X}+\mathsf{Y}$ & $2\mathsf{Y}$ & $2\mathsf{Y}$ & $\mathsf{Z}$ & $\mathsf{X}+\mathsf{Z}$ & $\mathsf{0}$ \\
& 79a & $\mathsf{X}$ & $2\mathsf{X}$ & $\mathsf{X}+\mathsf{Y}$ & $2\mathsf{Y}$ & $2\mathsf{Y}$ & $2\mathsf{Z}$ & $\mathsf{X}+\mathsf{Z}$ & $\mathsf{0}$ \\
& 79b & $\mathsf{X}$ & $2\mathsf{X}$ & $\mathsf{X}+\mathsf{Y}$ & $2\mathsf{Y}$ & $2\mathsf{Y}$ & $\mathsf{Y}+\mathsf{Z}$ & $\mathsf{X}+\mathsf{Z}$ & $\mathsf{0}$ \\
& 80a & $\mathsf{Y}$ & $2\mathsf{X}$ & $2\mathsf{X}$ & $2\mathsf{Y}$ & $\mathsf{Y}+\mathsf{Z}$ & $2\mathsf{Z}$ & $\mathsf{X}+\mathsf{Z}$ & $\mathsf{0}$ \\
& 80b & $\mathsf{Y}$ & $2\mathsf{X}$ & $2\mathsf{X}$ & $\mathsf{X}+\mathsf{Y}$ & $\mathsf{Y}+\mathsf{Z}$ & $2\mathsf{Z}$ & $\mathsf{X}+\mathsf{Z}$ & $\mathsf{0}$ \\
& 81a & $\mathsf{Y}$ & $2\mathsf{X}$ & $\mathsf{X}+\mathsf{Z}$ & $2\mathsf{Y}$ & $\mathsf{Y}+\mathsf{Z}$ & $2\mathsf{Z}$ & $2\mathsf{Z}$ & $\mathsf{0}$ \\
& 81b & $\mathsf{Y}$ & $2\mathsf{X}$ & $\mathsf{X}+\mathsf{Z}$ & $2\mathsf{Y}$ & $\mathsf{Y}+\mathsf{Z}$ & $2\mathsf{Z}$ & $2\mathsf{Z}$ & $\mathsf{Z}$ \\
& 82 & $\mathsf{Y}$ & $\mathsf{X}+\mathsf{Y}$ & $2\mathsf{X}$ & $\mathsf{Y}+\mathsf{Z}$ & $\mathsf{Y}+\mathsf{Z}$ & $\mathsf{Z}$ & $\mathsf{X}+\mathsf{Z}$ & $\mathsf{0}$ \\
& 83 & $\mathsf{Y}$ & $\mathsf{X}+\mathsf{Y}$ & $2\mathsf{X}$ & $\mathsf{Y}+\mathsf{Z}$ & $\mathsf{Y}+\mathsf{Z}$ & $2\mathsf{Z}$ & $\mathsf{X}+\mathsf{Z}$ & $\mathsf{0}$ \\
& 84 & $\mathsf{Y}$ & $\mathsf{X}+\mathsf{Y}$ & $2\mathsf{X}$ & $\mathsf{Y}$ & $\mathsf{Y}+\mathsf{Z}$ & $2\mathsf{Z}$ & $\mathsf{X}+\mathsf{Z}$ & $\mathsf{0}$ \\
& 85a & $\mathsf{Y}$ & $\mathsf{X}+\mathsf{Y}$ & $2\mathsf{X}$ & $2\mathsf{Y}$ & $\mathsf{Y}+\mathsf{Z}$ & $2\mathsf{Z}$ & $\mathsf{X}+\mathsf{Z}$ & $\mathsf{0}$ \\
& 85b & $\mathsf{Y}$ & $\mathsf{X}+\mathsf{Y}$ & $2\mathsf{X}$ & $\mathsf{X}+\mathsf{Y}$ & $\mathsf{Y}+\mathsf{Z}$ & $2\mathsf{Z}$ & $\mathsf{X}+\mathsf{Z}$ & $\mathsf{0}$ \\
& 86 & $\mathsf{Z}$ & $\mathsf{X}+\mathsf{Z}$ & $2\mathsf{X}$ & $\mathsf{Y}+\mathsf{Z}$ & $\mathsf{X}+\mathsf{Y}$ & $\mathsf{0}$ & $\mathsf{Y}+\mathsf{Z}$ & $\mathsf{X}+\mathsf{Y}$ \\
\cline{1-10}
\end{longtable}
\end{center}

We make some notes about expectional networks, or interesting groups of networks, in the list above. 

\begin{note}[The example of Wilhelm: 40a]
Network 40a in the list above is precisely the network identified by Wilhelm \cite{Wilhelm2009} as being capable of supercritical Hopf bifurcation. This network is analysed further in \cite[Section 5]{boros:hofbauer:2022b}. 
\end{note}

\begin{note}[A network with no species appearing on both sides of any reaction: 33a]
There is exactly one network in the list above, namely network 33a, where no species figures on both sides of any reaction. This network admits only a subcritical Hopf bifurcation. 
\end{note}

\begin{note}[A network where $L_1$ changes sign twice: 26]
\label{noteL1zerotwice}
Up to dynamical equivalence, there is only one network, namely network 26, where the Bautin set has two components (see Remark~\ref{remHcurve}). At all points on the Bautin set, the second Lyapunov coefficient, $L_2$, is negative.
\end{note}

\begin{note}[Networks which are equivalent in the sense of Example~\ref{exnondeq}]
\label{notenonsimpleequiv}
Recall that if we allow a wider range of transformations than species permutation, two CRNs which fail to be dynamically equivalent in the restrictive sense used here may still give rise to essentially the same dynamics. In particular, we find that the following nine pairs of networks give rise to the same set of differential equations following the recoordinatisation in Section~\ref{secrecoord}: 14 and 15; 17 and 18; 23 and 25; 28 and 32; 40 and 46; 42 and 47; 44 and 49; 67 and 70; 69 and 71. It is possible that there are further equivalences amongst the networks. 
\end{note}

\begin{note}[Networks with more than one stable limit set and/or more than one periodic orbit]
\label{noteBautin}
We list the networks which are claimed in Theorem~\ref{thmBautin} to permit multiple stable limit sets, or multiple periodic orbits, for some values of the rate constants. 
\begin{itemize}
\item The following $29$ networks admit rate constants where a linearly stable equilibrium coexists with a linearly stable periodic orbit and an unstable periodic orbit:
\begin{itemize}
\item The networks with $L_1\geq0$, namely networks 8--13;
\item The following networks where $L_1$ can take all signs: 6, 23--27, 54--57, and 59--71.
\end{itemize}
\item Networks 22 and 58 admit rate constants where an unstable equilibrium coexists with a linearly stable periodic orbit and an unstable periodic orbit. 
\end{itemize}
\end{note}

\begin{note}[Networks with two simple flow reactions of the form $0 \rightarrow \mathsf{X}_i$ or $\mathsf{X}_i \rightarrow 0$: 10a and 11a]
There are exactly two networks in the list above which include two reactions of the form $0 \rightarrow \mathsf{X}_i$ or $\mathsf{X}_i \rightarrow 0$, namely 10a and 11a. Although they do not admit supercritical Hopf bifurcation, these two networks nevertheless admit a stable, nondegenerate periodic orbit with mass action kinetics (see Theorem~\ref{thmBautin} and Note~\ref{noteBautin}) and could be considered the simplest of all bimolecular CRNs proven to do so. Remarkably, they also admit the coexistence of a stable periodic orbit and a stable equilibrium. As remarked on in Section~\ref{secconclusions}, the fully open extensions of these CRNs, namely 
\[
\mathsf{X} + \mathsf{Z} \rightarrow \mathsf{Y} + \mathsf{Z} \rightarrow 2\mathsf{Z}, \quad 0 \rightleftharpoons \mathsf{X}, \quad 0 \rightleftharpoons \mathsf{Y}, \quad 0 \rightleftharpoons \mathsf{Z}
\]
and
\[
\mathsf{X} + \mathsf{Z} \rightarrow 2\mathsf{Y}, \quad \mathsf{Y} + \mathsf{Z} \rightarrow 2\mathsf{Z}, \quad 0 \rightleftharpoons \mathsf{X}, \quad 0 \rightleftharpoons \mathsf{Y}, \quad 0 \rightleftharpoons \mathsf{Z}
\]
must admit a nondegenerate, stable periodic orbit with mass action kinetics. One of these was found in numerical simulations in \cite{banajiCRNosci}, while the other was missed.
\end{note}

\begin{note}[S-systems]
\label{noteSsystems}
$11$ out of the $86$ networks with nondegenerate Hopf bifurcation, namely, networks 5, 10, 14, 15, 37, 39, 40, 44, 46, 49, 50, give rise to so-called ``S-systems''. An S-system is a dynamical system on the positive orthant for which the right hand side is given by binomials. S-systems were introduced by Savageau \cite{savageau:1969a,savageau:1969b}, in the context of biochemical systems theory. For a recent review and an extensive list of references, see \cite{voit:2013}. As already observed in \cite{savageau:1969b}, the binomial structure of an S-system allows one to reduce the computation of positive equilibria to linear algebra.
In particular, it is easy to characterize when such a dynamical system has a unique positive equilibrium. At the same time, even a planar S-system with a unique positive equilibrium may give rise to rich dynamical behaviour, as demonstrated in \cite{boros:hofbauer:mueller:regensburger:2019,boros:hofbauer:2019}. Hopf bifurcations of planar S-systems are discussed in \cite{lewis:1991}, and stable limit cycles in these systems are constructed in \cite{yin:voit:2008}. 
\end{note}

\bibliographystyle{unsrt}


\end{document}